\documentclass[11pt,english]{amsart}
\usepackage{ae,aecompl}

\usepackage[T1]{fontenc}
\usepackage[latin9]{inputenc}
\usepackage{geometry}
\usepackage{color}
\usepackage{babel}
\usepackage{enumitem}
\usepackage{mathrsfs}
\usepackage{amstext}
\usepackage{amsthm}
\usepackage{amssymb}
\usepackage[unicode=true,pdfusetitle,
 bookmarks=true,bookmarksnumbered=false,bookmarksopen=false,
 breaklinks=false,pdfborder={0 0 1},backref=false,colorlinks=true]
 {hyperref}

\makeatletter
\numberwithin{equation}{section}
\numberwithin{figure}{section}
\theoremstyle{definition}
\ifx\thechapter\undefined
  \newtheorem{defn}{\protect\definitionname}
\else
  \newtheorem{defn}{\protect\definitionname}[chapter]
\fi
\theoremstyle{plain}
\ifx\thechapter\undefined
  \newtheorem{thm}{\protect\theoremname}
\else
  \newtheorem{thm}{\protect\theoremname}[chapter]
\fi
\theoremstyle{remark}
\ifx\thechapter\undefined
  \newtheorem{rem}{\protect\remarkname}
\else
  \newtheorem{rem}{\protect\remarkname}[chapter]
\fi
\theoremstyle{plain}
\ifx\thechapter\undefined
  \newtheorem{prop}{\protect\propositionname}
\else
  \newtheorem{prop}{\protect\propositionname}[chapter]
\fi
\theoremstyle{plain}
\ifx\thechapter\undefined
  \newtheorem{lem}{\protect\lemmaname}
\else
  \newtheorem{lem}{\protect\lemmaname}[chapter]
\fi

\usepackage{etex}
\usepackage{amsfonts}
\usepackage{amsthm}
\usepackage{array}
\usepackage{amsmath}
\usepackage{float}

\DeclareMathOperator {\esssup}{ess \ sup}

\date{}
\makeatletter
\makeatletter
\def\clearpage{%
  \ifvmode
    \ifnum \@dbltopnum =\m@ne
      \ifdim \pagetotal <\topskip
        \hbox{}
      \fi
    \fi
  \fi
  \newpage
  \thispagestyle{empty}
  \write\m@ne{}
  \vbox{}
  \penalty -\@Mi
}
\makeatother

\newcommand{\Rolosaysnopage}[1]{{ }}

\usepackage{pgf,tikz}
\usetikzlibrary{arrows}

\def\Xint#1{\mathchoice
{\XXint\displaystyle\textstyle{#1}}%
{\XXint\textstyle\scriptstyle{#1}}%
{\XXint\scriptstyle\scriptscriptstyle{#1}}%
{\XXint\scriptscriptstyle\scriptscriptstyle{#1}}%
\!\int}
\def\XXint#1#2#3{{\setbox0=\hbox{$#1{#2#3}{\int}$}
\vcenter{\hbox{$#2#3$}}\kern-.5\wd0}}

\def\dashint{\Xint-}

\makeatletter
\newsavebox{\@brx}
\newcommand{\llangle}[1][]{\savebox{\@brx}{\(\m@th{#1\langle}\)}%
  \mathopen{\copy\@brx\kern-0.5\wd\@brx\usebox{\@brx}}}
\newcommand{\rrangle}[1][]{\savebox{\@brx}{\(\m@th{#1\rangle}\)}%
  \mathclose{\copy\@brx\kern-0.5\wd\@brx\usebox{\@brx}}}
\makeatother

\newcommand{\Cn}{\ensuremath{\mathbb{C}^n}}

\newcommand{\D}{\ensuremath{\mathscr{D}}}

\newcommand{\MC}[1]{\ensuremath{\mathcal{#1}}}

\newcommand{\V}[1]{\ensuremath{\vec{#1}}}

\newcommand{\innp}[1]{\left< #1 \right>}

\newcommand{\W}{\ensuremath{{\mathcal W}}}

\newcommand{\norm}[1]{\ensuremath{\left\|#1\right\|}}
\newcommand{\pr}[1]{\ensuremath{\left(#1\right)}}

\newcommand{\Ctwon}{\ensuremath{\mathbb{C}^{2n}}}

\makeatother

\providecommand{\definitionname}{Definition}
\providecommand{\lemmaname}{Lemma}
\providecommand{\propositionname}{Proposition}
\providecommand{\remarkname}{Remark}
\providecommand{\theoremname}{Theorem}

\begin{document}
\global\long\def\essssup{\text{ess sup}}

\global\long\def\esssinf{\text{ess inf}}

\title{Sharp $A_{1}$ weighted estimates for vector valued operators}

\author{Joshua Isralowitz}

\address{Department of Mathematics, University at Albany, 1400 Washington
Ave., Albany, NY 12159, USA E-mail address: }

\email{jisralowitz@albany.edu}

\author{Sandra Pott}

\address{Centre for Mathematical Sciences, University of Lund, Lund, Sweden.}

\email{sandra@maths.lth.se}

\author{Israel P. Rivera-Ríos }

\address{Universidad Nacional del Sur - Instituto de Matemática de Bahía Blanca,
Bahía Blanca, Argentina.}

\email{israel.rivera@uns.edu.ar}
\begin{abstract}
Given $1\leq q<p<\infty$, quantitative weighted $L^{p}$ estimates,
in terms of $A_{q}$ weights, for vector valued maximal functions,
Calderón-Zygmund operators, commutators and maximal rough singular
integrals are obtained. The results for singular operators will rely
upon suitable convex body domination results, which in the case of
commutators will be provided in this work, obtaining as a byproduct
a new proof for the scalar case as well.
\end{abstract}

\maketitle

\section{Introduction}

We recall that a weight, namely, a non negative locally integrable
function $w$ belongs to $A_{p}$ for $1<p<\infty$ if
\[
[w]_{A_{p}}=\sup_{Q}\frac{1}{|Q|}\int_{Q}w(x)dx\left(\frac{1}{|Q|}\int_{Q}w(y)^{-\frac{p'}{p}}dy\right)^{\frac{p}{p'}}<\infty.
\]
The $A_{p}$ class of weights characterizes the $L^{p}(w)$ boundedness
of the maximal function as B. Muckenhoupt established in the 70s.
Subsequent works of B. Muckenhoupt himself R. Wheeden, R. Hunt, R. R. Coifman
and C. Fefferman were devoted to explore the connection of the $A_{p}$ class with
weighted estimates for singular integrals. However, it was not until
the 2000s that the quantitative dependence on the so called $A_{p}$
constant, namely $[w]_{A_{p}}$, became a trending topic. Probably
the paradigmatic question in that line of research was the $A_{2}$ theorem finally
established by T. Hytönen \cite{H}.

Now we recall that the $A_{p}$ classes are increasing, so it is natural
to define $A_{\infty}=\bigcup_{p\geq1}A_{p}$. T. Hytönen and C. Pérez
\cite{HPAinfty} proved that
\[
[w]_{A_{\infty}}=\sup_Q\frac{1}{w(Q)}\int_{Q}M(w\chi_{Q})(x)dx<\infty
\]
is the smallest constant characterizing $A_{\infty}$, at least up
until now, and provided a number of quantitative estimates in terms
of $[w]_{A_{\infty}}$. Nevertheless, it is worth noting that essentially
the same constant had already appeared in works by Fujii \cite{F}
and Wilson \cite{W}. After \cite{HPAinfty} several papers have been
devoted to the study of quantitative weigthed estimates in terms of
the $A_{p}$ and the $A_{\infty}$ constants.

Among the possible extensions of the classical scalar theory of Calderón-Zygmund
operators, vector valued extensions have received an increasing degree
of attention in the last years. Let $W:\mathbb{R}^{d}\rightarrow\mathbb{C}^{n\times n}$
a matrix weight, namely, a matrix function such that $W(x)$ is positive
definite a.e. Given $f:\mathbb{R}^{d}\rightarrow \mathbb{C}^{n}$ and
$1<p<\infty$, we define
\[
\|f\|_{L^{p}(W)}=\left(\int_{\mathbb{R}^{d}}\left|W^{\frac{1}{p}}(x)f(x)\right|^{p}dx\right)^{\frac{1}{p}}.
\]
Let $1<p<\infty$. We say that a matrix weight $W$ is an $A_{p}$
weight if
\[
[W]_{A_{p}}=\sup_{Q}\frac{1}{|Q|}\int_{Q}\left(\frac{1}{|Q|}\int_{Q}\left\Vert W^{\frac{1}{p}}(x)W^{-\frac{1}{p}}(y)\right\Vert ^{p'}dy\right)^{\frac{p}{p'}}dx<\infty.
\]
Matrix $A_{p}$ weights were introduced by S. Treil and A. Volberg
in \cite{TV}. In the late 90s it was shown in a series of works by
M. Goldberg \cite{G}, F. Nazarov and S. Treil \cite{NT} and A. Volberg
\cite{V} that if $W$ is a matrix $A_{p}$ weight and $T$ is a Calderón-Zygumund
operator, then $T$ is bounded on $L^{p}(W)$. The definition of $A_p$ that we have presented here is due to S. Roudenko \cite{R} and is equivalent to the definitions in the aforementioned works.

Contrary to what happpens in the scalar setting, the $A_2$ conjecture remains an open problem in the vector valued setting. In \cite{BPW}, K. Bickel, S. Petermichl and B. Wick proved that the dependence of the norm of the martingale and Hilbert
transform on the $A_{2}$ constant of the weight $W$ is at most $[W]_{A_{2}}^{\frac{3}{2}}\log([W]_{A_{2}})$.
The second author and A. Stoica \cite{PS} established that the dependence
of the norm of all Calderón-Zygmund operators with cancellation on
$[W]_{A_{2}}$ coincides with the one for the matrix martingale transform,
hence reducing the $A_{2}$ conjecture for those operators to the
proof of the linear bound for the latter.

Given $1\leq q<\infty$, we say that $W\in A_{q,\infty}^{sc}$ if
\[
[W]_{A_{q,\infty}^{sc}}=\sup_{e\in\mathbb{C}^{n}}\left[\left|W^{\frac{1}{q}}e\right|^{q}\right]_{A_{\infty}}<\infty.
\]
Quite recently F. Nazarov, S. Petermichl, S. Treil and A. Volberg
\cite{NPTV} established the following quantitative estimate for $W\in A_{2}$,
\begin{equation}
\|T\vec{f}\|_{L^{2}(W)}\leq c_{n,d,T}[W]_{A_{2}}^{\frac{1}{2}} [W]_{A_{2,\infty}^{sc}}^\frac12 [W^{-1}]_{A_{2,\infty}^{sc}} ^\frac12\|\vec{f}\|_{L^{2}(W)}\leq c_{n,T}[W]_{A_{2}}^{\frac{3}{2}}\|f\|_{L^{2}(W)}.\label{eq:A2NPTV}
\end{equation}
The preceding estimate is obtained using the so called convex body
domination. In that work the linear dependence on the $A_{2}$ constant
is conjectured. In the case of maximal rough singular integrals with
$\Omega\in L^{\infty}(\mathbb{S}^{n-1})$, the following estimate,
in the case $p=2$, was quite recently provided by F. Di Plinio, K.
Li and T. Hytönen \cite{DiPHL},
\[
\left\Vert \sup_{\delta>0}|W^{\frac{1}{2}}T_{\Omega,\delta}\vec{f}|\right\Vert _{L^{2}(\mathbb{R}^{d})}\leq c_{n,d,T}[W]_{A_{2}}^{\frac{5}{2}}\|W^{\frac{1}{2}}\vec{f}\|_{L^{2}(\mathbb{R}^{d})}
\]
where the scalar operator $T_{\Omega,\delta}$ is defined as follows
\[
T_{\Omega,\delta}f(x)=\int_{|x-y|>\delta}\frac{\Omega\left(\frac{x-y}{|x-y|}\right)}{|x-y|^{d}}f(y)dy.
\]
Very recently D. Cruz-Uribe, J. Isralowitz and K. Moen \cite{CUIM}
extended \eqref{eq:A2NPTV} to every $1<p<\infty$, providing the
following estimate
\[
\|T\vec{f}\|_{L^{p}(W)}\leq c_{n,d,T}\left[W\right]_{A_{p}}\left[W^{-\frac{p'}{p}}\right]_{A_{p',\infty}^{sc}}^{\frac{1}{p}}\left[W\right]_{A_{p,\infty}^{sc}}^{\frac{1}{p'}}\|f\|_{L^{p}(W)}\leq c_{n,d,T}\left[W\right]_{A_{p}}^{1+\frac{1}{p-1}-\frac{1}{p}}\|\V{f}\|_{L^{p}(W)}.
\]

Some sharp estimates have been obtained as well in the vector valued
setting. T. Hytönen, S. Petermichl and A. Volberg \cite{HPV}, and
Isralowitz, Kwon, and the first author \cite{IKP} established the
linear upper bound on $[W]_{A_{2}}$ for the matrix-weighted square
function and the matrix-weighted maximal function, respectively (namely, $M_{W,p}$ defined as in Section \ref{sec:MR}).

We recall that given a linear operator $G$ and a locally integrable
function $b$, the commutator $[b,G]$ is defined by
\[
[b,G]f(x)=b(x)Gf(x)-G(bf)(x).
\]

At this point we turn our attention back to the scalar setting. A.
Lerner, S. Ombrosi and C. Pérez \cite{LOP2,LOP} established the following
result for Calderón-Zygmund operators. Given a Calderón-Zygmund operator
$T$ and $w\in A_{1}$ we have that
\begin{equation}
\|Tf\|_{L^{p}(w)}\leq c_{n,T}pp'[w]_{A_{1}}^{\frac{1}{p}}[w]_{A_{\infty}}^{\frac{1}{p'}}\|f\|_{L^{p}(w)}.\label{eq:A1escalar}
\end{equation}
In the case of commutators, for $b\in BMO$ and $T$ a Calderón-Zygmund
operator, C. Ortiz-Caraballo \cite{OC} proved that
\begin{equation}
\|[b,T]f\|_{L^{p}(w)}\leq c_{n,T}\|b\|_{BMO}\left(pp'\right)^{2}[w]_{A_{1}}^{\frac{1}{p}}[w]_{A_{\infty}}^{1+\frac{1}{p'}}\|f\|_{L^{p}(w)}.\label{eq:A1ScalarComm}
\end{equation}
One of the motivations to obtain such a precise estimate for Calderón-Zygmund
operators was to provide a proof of the $A_{2}$ constant. Assume
that for every $w\in A_{1}$
\[
\|Tf\|_{L^{1,\infty}(w)}\leq c\varphi([w]_{A_{1}})\|f\|_{L^{1}(w)}.
\]
Then we also have that for every $1<p<\infty$ and every $w\in A_{p}$
(\cite{LOP2})
\[
\|Tf\|_{L^{p,\infty}(w)}\leq c\varphi([w]_{A_{p}})\|f\|_{L^{p}(w)}.
\]
We observe that in \cite{LOP2} it was proved that $\varphi(t)\leq t\log(e+t)$
using \eqref{eq:A1escalar} as a main ingredient and it was also conjectured
that $\varphi(t)\simeq t$. Being true the latter would have led to
a proof of the $A_{2}$ conjecture, since in \cite{PTV} it was established
that
\[
\|T\|_{L^{2}(w)}\leq c_{n,T}[w]_{A_{2}}+c_{n,T}\left(\|T\|_{L^{2}(w)\rightarrow L^{2,\infty}(w)}+\|T\|_{L^{2}(w^{-1})\rightarrow L^{2,\infty}(w^{-1})}\right).
\]
However, the fact that $\varphi(t)\simeq t$ was disproved in \cite{NRVV},
furthermore, in \cite{LNO} it was established that $\varphi(t)\simeq t\log(e+t)$,
and consequently the estimate in \cite{LOP2} is sharp.

\section{Main Results}\label{sec:MR}

One of the main purposes of this paper is to provide vector valued
counterparts of \eqref{eq:A1escalar} and \eqref{eq:A1ScalarComm}.
To provide that kind of estimates we rely upon the definition of the
matrix $A_{1}$ class  that M. Frazier and S. Roudenko introduced
in \cite{FR}.
\begin{defn}
We say that a weight $W\in A_{1}$ if
\[
[W]_{A_{1}}=\sup_{Q}\esssup_{y\in Q}\frac{1}{|Q|}\int_{Q}\|W(x)W^{-1}(y)\|dx<\infty
\]
\end{defn}
Before presenting our first result we would like to discuss briefly
the definition of the maximal function. Due to the non-linearity of
the maximal function, when it comes to study weighted estimates for
it, the approach that has been mainly considered in the literature
is to study weighted variants of it (see \cite{G,IKP}). In what
follows we will deal with the following weighted maximal
functions.
\[
M_{W,p}(\vec{f})(x)=\sup_{x\in Q}\frac{1}{|Q|}\int_{Q}\left|W^{\frac{1}{p}}(x)W^{-\frac{1}{p}}(y)\vec{f}(y)\right|dy
\]
\[
M'_{W,p}(\vec{f})(x)=\sup_{x\in Q}\frac{1}{|Q|}\int_{Q}\left|\mathcal{W}_{Q,p}W^{-\frac{1}{p}}(y)\vec{f}(y)\right|dy
\]
We remit the reader to Section \ref{sec:Apweights} for the definition
of $\mathcal{W}_{Q}$.
\begin{thm}
\label{Thm:A1Mat} Let $W\in A_{1}$ and $1<p<\infty$. Then
\begin{eqnarray}
\|M_{W,p}\|_{L^{p}(\mathbb{R}^{d};\mathbb{C}^{n})\rightarrow L^{p}(\mathbb{R}^{d})} & \leq & c_{n,p}[W]_{A_{1}}^{\frac{1}{p}}\label{eq:A1MW}\\
\|M'_{W,p}\|_{L^{p}(\mathbb{R}^{d};\mathbb{C}^{n})\rightarrow L^{p}(\mathbb{R}^{d})} & \leq & c_{n,p}[W]_{A_{1}}^{\frac{1}{p}}\label{eq:A1MprimaW}
\end{eqnarray}
Let $T$ a Calderón-Zygmund operator, $b\in BMO$ and $\Omega\in L^{\infty}(\mathbb{S}^{d-1})$
with $\int_{\mathbb{S}^{d-1}}\Omega=0$, if $W\in A_{1}$ and $1<p<\infty$,
\begin{eqnarray}
 &  & \|T\|_{L^{p}(W)\rightarrow L^{p}(W)}\leq c_{n,p,T}[W]_{A_{1}}^{\frac{1}{p}}[W]_{A_{1,\infty}^{sc}}^{\frac{1}{p'}}\leq c_{n,p,T}[W]_{A_{1}}\label{eq:A1Mat}\\
 &  & \left\Vert T_{\Omega,W,p}^{*}\right\Vert _{L^{p}(\mathbb{R}^{d}; \mathbb{C}^{n})\rightarrow L^{p}(\mathbb{R}^{d})}\leq c_{n,d,\Omega,p}[W]_{A_{1}}^{\frac{1}{p}}[W]_{A_{1,\infty}^{sc}}^{1+\frac{1}{p'}}\leq c_{n,d,\Omega,p}[W]_{A_{1}}^{2}\label{eq:A1Rough}\\
 &  & \|[b,T]\|_{L^{p}(W)\rightarrow L^{p}(W)}\leq c_{n,p,T}\|b\|_{BMO}[W]_{A_{1}}^{\frac{1}{p}}[W]_{A_{1,\infty}^{sc}}^{1+\frac{1}{p'}}\leq c_{n,p,T}[W]_{A_{1}}^{2}\label{eq:A1MatComm}
\end{eqnarray}
where $T_{\Omega,W,p}^{*}f=\sup_{\delta>0}\left|W^{\frac{1}{p}}T_{\Omega,\delta}\left(W^{-\frac{1}{p}}\vec{f}\right)\right|$.
\end{thm}

Coming back once again to the scalar setting, it is a known fact that
an extrapolation argument \cite[Corollary 4.3]{D} allows to prove
that if we have that
\[
\|Gf\|_{L^{p}(w)}\leq c_{T,p,q}\varphi([w]_{A_{1}})\|f\|_{L^{p}(w)},
\]
for every $A_{1}$ weight, then the same dependence holds as well
for every $w\in A_{q}$ with $1\leq q<p$, namely,
\[
\|Gf\|_{L^{p}(w)}\leq c_{T,p,q}\varphi([w]_{A_{q}})\|f\|_{L^{p}(w)},
\]

Extrapolation arguments, in case of being feasible, have not been
developed yet in this setting so we provide a direct proof of the
preceding result in the cases considered in Theorem \ref{Thm:A1Mat}.
We observe that we recover again the linear dependence already available
in the scalar case. We wonder whether it is possible to provide some
estimate analogous to the one supremmum estimates obtained in \cite{Li}
and \cite{RRio}.
\begin{thm}
\label{Thm:AqMat}Let $1<q<p<\infty$ and $W\in A_{q}$. Then
\begin{eqnarray}
\|M_{W,p}\|_{L^{p}(\mathbb{R}^{d};\mathbb{C}^{n})\rightarrow L^{p}(\mathbb{R}^{d})} & \leq & c_{n,p,T}[W]_{A_{q}}^{^{\frac{1}{p}}}\label{eq:AqMW}\\
\|M'_{W,p}\|_{ L^{p}(\mathbb{R}^{d};\mathbb{C}^{n})\rightarrow L^{p}(\mathbb{R}^{d})} & \leq & c_{n,p,T}[W]_{A_{q}}^{\frac{1}{p}}\label{eq:AqMprimaW}
\end{eqnarray}
Let $T$ a Calderón-Zygmund operator, $b\in BMO$ and $\Omega\in L^{\infty}(\mathbb{S}^{d-1})$
with $\int_{\mathbb{S}^{d-1}}\Omega=0$, if $W\in A_{1}$ and $1<p<\infty$,
\begin{eqnarray}
 &  & \|T\|_{L^{p}(W)\rightarrow L^{p}(W)}\leq c_{n,p,q,T}[W]_{A_{q}}^{\frac{1}{p}}[W]_{A_{q,\infty}^{sc}}^{\frac{1}{p'}}\leq c_{n,p,T}[W]_{A_{q}}\label{eq:AqMat}\\
 &  & \left\Vert T_{\Omega,W,p}^{*}\right\Vert _{L^{p}(\mathbb{R}^{d};\mathbb{C}^{n})\rightarrow L^{p}(\mathbb{R}^{d})}\leq c_{n,d,\Omega,p}[W]_{A_{q}}^{\frac{1}{p}}[W]_{A_{q,\infty}^{sc}}^{1+\frac{1}{p'}}\leq c_{n,d,\Omega,p}[W]_{A_{q}}^{2}\label{eq:AqRough}\\
 &  & \|[b,T]\|_{L^{p}(W)\rightarrow L^{p}(W)}\leq c_{n,p,q,T}\|b\|_{BMO}[W]_{A_{q}}^{\frac{1}{p}}[W]_{A_{q,\infty}^{sc}}^{1+\frac{1}{p'}}\leq c_{n,p,q,T}\|b\|_{BMO}[W]_{A_{q}}^{2}\label{eq:AqMatComm}
\end{eqnarray}
where $T_{\Omega,W,p}^{*}f=\sup_{\delta>0}\left|W^{\frac{1}{p}}T_{\Omega,\delta}\left(W^{-\frac{1}{p}}f\right)\right|$.
\end{thm}
We would like to note that both in Theorems \ref{Thm:A1Mat} and \ref{Thm:AqMat}
the dependences obtained are the same as the best known ones in the
scalar case, and therefore, besides the case of the maximal rough singular integral, sharp.

The rest of the paper is organized as follows. In Section \ref{sec:CBD}
we present a convex body domination result for commutators. We provide
some extra facts about matrix $A_{p}$ weights in Section \ref{sec:Apweights}.
Finally, in Section \ref{sec:Proofs of A1Estimates} we settle Theorems
\ref{Thm:A1Mat} and \ref{Thm:AqMat}.

\section{Convex body domination for Commutators \label{sec:CBD}}

We begin the section borrowing some definitions from \cite{LN}. We
say that a family of cubes $\mathcal{D}$ is a dyadic lattice if it
satisfies the following properties
\begin{enumerate}
\item If $Q\in\mathcal{D}$ every dyadic child of $Q$ belongs to $\mathcal{D}$.
In other words, if $\mathcal{D}(Q)$ is the standard grid of dyadic
cubes of $Q$ and $Q\in\mathcal{D}$ then $\mathcal{D}(Q)\subseteq\mathcal{D}$.
\item If $Q_{1},Q_{2}\in\mathcal{D}$ there exists a common ancestor in
$\mathcal{D}$ that is there exists $Q\in\mathcal{D}$ such that $Q_{1},Q_{2}\in\mathcal{D}(Q)$.
\item For every compact set $K\subseteq\mathbb{R}^{d}$ there exists $Q\in\mathcal{D}$
such that $K\subseteq Q$.
\end{enumerate}
Given $\eta\in(0,1)$ we say that $\mathcal{S}\subset\D$
is a $\eta$-sparse family if for every $Q\in\mathcal{S}$ there exists
a measurable subset $E_{Q}\subset Q$ such that
\begin{enumerate}
\item $\eta|Q|\leq|E_{Q}|.$
\item The sets $E_{Q}$ are pairwise disjoint.
\end{enumerate}
Further, given $\Lambda > 1$ we say that $\mathcal{S} \subset \D$
is a $\Lambda$ Carleson family if for every $Q \in \MC{S}$,
$$\sum_{P \in \MC{S}, P \subseteq Q} |P| \leq \Lambda |Q|.$$
Clearly every $\eta$-sparse family is $\eta^{-1}$ Carleson,
since
$$\sum_{P \in \MC{S}, P \subseteq Q} |P| \leq \eta^{-1} \sum_{P \in
  \MC{S}, P \subseteq Q} |E_P| \leq \Lambda^{-1} |Q|.$$
Though less obvious, the converse is true. Every $\Lambda$ Carleson
family is $\Lambda^{-1}$ sparse~\cite[Lemma 6.3]{LN}.  We will also
use without further comment the fact that every $\Lambda$ Carleson
family can be written as a union of $m$ Carleson families, each of
which is $1 + \frac{\Lambda-1}{m}$ Carleson~\cite[Lemma 6.6]{LN}.
Hereafter we will sometimes refer to a family as sparse or Carleson
without reference to $\eta$ or $\Lambda$  if the
specific values of these constants are unimportant.

Convex body domination was introduced by F. Nazarov, S. Petermichl
and A. Volberg in \cite{NPTV}. That notion provides a suitable counterpart
to sparse domination in the vector-valued setting. Let $f:\mathbb{R}^{d}\longrightarrow\mathbb{C}^{n}$.
Given a cube $Q$ if additionally $f\in L^{r}(Q)$ where $1\leq r<\infty$
and $r'=\infty$ if $r=1$, we define
\[
\langle\langle\vec{f}\rangle\rangle_{r,Q}=\left\{ \frac{1}{|Q|}\int_{Q}f\varphi dx\,:\,\varphi:Q\rightarrow\mathbb{R},\,\varphi\in B_{L^{r'}}(Q)\right\}
\]
where $B_{L^{r'}(Q)}=\left\{ \phi\in L^{r'}(Q)\,:\,\|\phi\|_{L^{r'}}\leq1\right\} $.
We will drop the subscript $r$ in the case $r=1$. In \cite{NPTV}
it was established that $\langle\langle f\rangle\rangle_{Q}$ is a
symetric, convex and compact set in $\mathbb{C}^{n}$ and in \cite{DiPHL}
that property was extended to the case $r>1$.

We recall that given $T$ a linear operator, the grand-maximal operator
$M_{T}$ was defined for first as follows in \cite{Le}
\[
M_{T}f(x)=\sup_{Q\ni x}\essssup_{y\in Q}|T(f\chi_{
{ \mathbb{R}^{d}}\setminus3Q})(y)|.
\]
In \cite{NPTV} the authors proved the following result (see also
\cite{HNotes}).
\begin{thm}
\label{Thm:ConvexBody}Let $T:L^{1}(\mathbb{R}^{d})\rightarrow L^{1,\infty}(\mathbb{R}^{d})$
be a linear operator such that also $M_{T}:L^{1}(\mathbb{R}^{d})\rightarrow L^{1,\infty}(\mathbb{R}^{d})$.
For $\vec{f}\in L_{c}^{\infty}(\mathbb{R}^{d};{\mathbb{C}^{n}})$ and
$\varepsilon\in(0,1)$ there exists a $(1-\varepsilon)$-sparse collection
of dyadic cubes such that
\[
T\vec{f}(x)\in\frac{c_{d,n}c_{T}}{\varepsilon}\sum_{j=1}^{3^{n}}\sum_{Q\in\mathcal{S}_{j}}\langle\langle\vec{f}\rangle\rangle_{Q}\chi_{Q}(x)
\]
where $c_{T}=\|T\|_{L^{1}\rightarrow L^{1,\infty}}+\|M_{T}\|_{L^{1}\rightarrow L^{1,\infty}}$.
 More precisely, there exist functions $k_{Q}\in B_{L^{\infty}(Q\times Q)}$
such that
\begin{equation}
T\vec{f}(x)=\frac{c_{d,n}c_{T}}{\varepsilon}\sum_{j=1}^{{  3^{d}}}\sum_{Q\in\mathcal{S}_{j}}\left(\frac{1}{|Q|}\int_{Q}k_{Q}(x,y)\vec{f}(y)dy\right)\chi_{Q}(x).\label{eq:Sparse}
\end{equation}
\end{thm}
Our purpose in this section is to establish the following vector-valued
counterpart for commutators extending \cite[Theorem 1.1]{LORR}.

\begin{thm}
\label{Thm:ConvexBodyComm}Let $T:L^{1}(\mathbb{R}^{d})\rightarrow L^{1,\infty}(\mathbb{R}^{d})$
be a linear operator such that also $M_{T}:L^{1}(\mathbb{R}^{d})\rightarrow L^{1,\infty}(\mathbb{R}^{d})$.
For $f\in L_{c}^{\infty}(\mathbb{R}^{d};{  \mathbb{C}^{n}})$, every $b\in L_{loc}^{1}$
and $\varepsilon\in(0,1)$ there exists a $(1-\varepsilon)$-sparse
collection of dyadic cubes such that
\[
[b,T]\vec{f}(x)\in\frac{c_{d,n}c_{T}}{\varepsilon}\sum_{j=1}^{{  3^{d}}}\sum_{Q\in\mathcal{S}_{j}}\left[(b(x)-\langle b\rangle_{Q})\langle\langle\vec{f}\rangle\rangle_{Q}\chi_{Q}(x)+\langle\langle(b-b_{Q})\vec{f}\rangle\rangle_{Q}\chi_{Q}(x)\right]
\]
where each { $c_{d, n}$ is a constant depending on $n$ and $d$ }and $c_{T}=\|T\|_{L^{1}\rightarrow L^{1,\infty}}+\|M_{T}\|_{L^{1}\rightarrow L^{1,\infty}}$.
 More precisely, there exist functions $k_{Q},\,k_{Q}^{*}\in B_{L^{\infty}(Q\times Q)}$
such that
\[
\begin{split}[b,T]\vec{f}(x) & =\frac{c_{d,n}c_{T}}{\varepsilon}\sum_{j=1}^{{  3^{d}}}\sum_{Q\in\mathcal{S}_{j}}(b(x)-\langle b\rangle_{Q})\left(\frac{1}{|Q|}\int_{Q}k_{Q}(x,y)\vec{f}(y)dy\right)\chi_{Q}(x)\\
 & {  - } \frac{c_{d,n}c_{T}}{\varepsilon}\sum_{j=1}^{{  3^{d}}}\sum_{Q\in\mathcal{S}_{j}}\left(\frac{1}{|Q|}\int_{Q}k_{Q}^{*}(x,y)(b(y)-b_{Q})\vec{f}(y)dy\right)\chi_{Q}(x)
\end{split}
\]
\end{thm}
We observe that A. Lerner \cite{Le} proved for Calderón-Zygmund operators
that
\[
\|M_{T}\|_{L^{1}\rightarrow L^{1,\infty}}\leq c_{d}\left(\|T\|_{L^{2}\rightarrow L^{2}}+c_{K}+\|\omega\|_{\text{Dini}}\right)
\]
and it is also a known fact that
\[
\|T\|_{L^{1}\rightarrow L^{1,\infty}}\leq c_{d}\left(\|T\|_{L^{2}\rightarrow L^{2}}+\|\omega\|_{\text{Dini}}\right).
\]
Consequently Theorems \ref{Thm:ConvexBody} and \ref{Thm:ConvexBodyComm}
hold in the case that $T$ is a Calderón-Zygmund operator with
\[
c_{T}=\|T\|_{L^{2}\rightarrow L^{2}}+c_{K}+\|\omega\|_{\text{Dini}}.
\]

\subsection{Proof of Theorem \ref{Thm:ConvexBodyComm}}

$f\in L_{c}^{\infty}(\mathbb{R}^{d};{  \mathbb{C}^{n}})$ and $b\in L_{loc}^{1}$,
we further assume that $b\in L^{\infty}.$ We define the $\Ctwon$
valued function $\tilde{f}$ by
\[
\tilde{f}(x)=\begin{pmatrix}\vec{f}(x)\\
\vec{f}(x)
\end{pmatrix}
\]
and define the $2n\times2n$ block matrix $\Phi(x)$ by
\[
\Phi(x)=\begin{pmatrix}1_{n\times n} & b\otimes1_{n\times n}\\
0 & 1_{n\times n}
\end{pmatrix}
\]
so that
\[
\Phi^{-1}(x)=\begin{pmatrix}1_{n\times n} & -b\otimes1_{n\times n}\\
0 & 1_{n\times n}
\end{pmatrix}.
\]
Then we have that
\[
\Phi^{-1}(y)\tilde{f}(y)=\begin{pmatrix}1_{n\times n} & -b(y)\otimes1_{n\times n}\\
0 & 1_{n\times n}
\end{pmatrix}\begin{pmatrix}\vec{f}(y)\\
\vec{f}(y)
\end{pmatrix}=\begin{pmatrix}\vec{f}(y)-\vec{f}(y)b(y)\\
\vec{f}(y)
\end{pmatrix}
\]
By assumption, $\Phi^{-1}\tilde{f}$ is bounded with compact support.
A direct computation shows that
\[
\Phi(x)(T\Phi^{-1}\tilde{f})(x)=\begin{pmatrix}T\vec{f}(x)+[b,T]\vec{f}(x)\\
T\vec{f}(x)
\end{pmatrix}
\]
Lets plug in $\Phi^{-1}(y)\tilde{f}(y)$ into \eqref{eq:Sparse} and
equate components. Namely
$$\Phi^{-1} (y) \tilde{f}(y) = \begin{pmatrix}1_{n \times n} & - b(y) \otimes 1_{n \times n}   \\ 0 & 1_{n \times n}  \end{pmatrix} \begin{pmatrix} f(y) \\ f(y) \end{pmatrix} = \begin{pmatrix} f(y) - f(y) b(y)  \\ f(y) \end{pmatrix} $$
 so that
\begin{align*}
\Phi(x)(T\Phi^{-1}\tilde{f})(x) & =c_{d,n}c_{T}\sum_{j=1}^{3^{d}}\sum_{Q\in\MC{S}_{j}}\Phi(x)\begin{pmatrix}\langle k_{Q}(x,\cdot)(f-fb)\rangle_{Q}\\
\langle k_{Q}(x,\cdot)f\rangle_{Q}
\end{pmatrix}\chi_{Q}(x)\\
 & =c_{d,n}c_{T}\sum_{j=1}^{3^{d}}\sum_{Q\in\MC{S}_{j}}\begin{pmatrix}\langle k_{Q}(x,\cdot)(f-fb)\rangle_{Q}+b(x)\langle k_{Q}(x,\cdot)f\rangle_{Q}\\
\langle k_{Q}(x,\cdot)f\rangle_{Q}
\end{pmatrix}\chi_{Q}(x)
\end{align*}
However, adding and subtracting $\langle k_{Q}(x,\cdot)f\rangle_{Q}\langle b\rangle_{Q}$
to the first component, we get
\[
\begin{split} & \Phi(x)(T\Phi^{-1}\tilde{f})(x)\\
 & =c_{d,n}c_{T}\sum_{j=1}^{3^{d}}\sum_{Q\in\MC{S}_{j}}\begin{pmatrix}\langle k_{Q}(x,\cdot)(f-fb)\rangle_{Q}+b(x)\langle k_{Q}(x,\cdot)f\rangle_{Q}\\
\langle k_{Q}(x,\cdot)f\rangle_{Q}
\end{pmatrix}\chi_{Q}(x)\\
 & =c_{d,n}c_{T}\sum_{j=1}^{3^{d}}\sum_{Q\in\MC{S}_{j}}\begin{pmatrix}\langle k_{Q}(x,\cdot)f\rangle_{Q}+\langle k_{Q}(x,\cdot)f(\langle b\rangle_{Q}-b)\rangle_{Q}+(b(x)-\langle b\rangle_{Q})\langle k_{Q}(x,\cdot)\vec{f}\rangle_{Q}\\
\langle k_{Q}(x,\cdot)\vec{f}\rangle_{Q}
\end{pmatrix}\chi_{Q}(x)
\end{split}
\]

Hence,
\begin{align*}
[b,T]\vec{f}(x) & =c_{d,n}c_{T}\sum_{j=1}^{3^{d}}\sum_{Q\in\MC{S}_{j}}\dashint_{Q}k_{Q}(x,y)(b(y)-\langle b\rangle_{Q})\vec{f}(y)\,dy\\
 & +(b(x)-\langle b\rangle_{Q})\pr{\dashint_{Q}f(y)k_{Q}(x,y)\vec{f}(y)\,dy}
\end{align*}
and we are done.
\begin{rem}
The proof presented above works as well in the case $n=1$, hence
providing a new proof for the scalar case that was settled in \cite{LORR}.
\end{rem}

\section{The reverese Hölder inquality. $A_{1},\,A_{q}$ and $A_{q,\infty}^{sc}$
weights}\label{sec:Apweights}

%
We recall that if $\rho(x)$ is a norm on ${  \mathbb{C}^{n}}$ there
exists a positive matrix $A$, that we call reducing operator of $\rho$
such that
\[
\rho(x)\simeq|Ax|\qquad x\in{  \mathbb{C}^{n}}.
\]
If $1\leq p<\infty$ we will call $\mathcal{W}_{Q,p}$ the reducing
operator for
\[
\rho_{W,p,Q}(x)=\left(\frac{1}{|Q|}\int_{Q}\left|W^{\frac{1}{p}}(t)x\right|^{p}dt\right)^{\frac{1}{p}}.
\]
In the case $1<p<\infty$ we shall call $\mathcal{W}'_{Q,p}$ the reducing
operator for
\[
\rho_{W,p',Q}^{*}(x)=\left(\frac{1}{|Q|}\int_{Q}\left|W^{-\frac{1}{p}}(t)x\right|^{p'}dt\right)^{\frac{1}{p'}}.
\]
It follows from the proof of Roudenko's characterization \cite[Lemma 1.3]{R}
that
\[
[W]_{A_{p}}\simeq\|\mathcal{W}_{Q,p}\mathcal{W}'_{Q,p}\|^{p}\qquad1<p<\infty.
\]
Now we observe that if we call $V=W^{-\frac{1}{p-1}}$, we have that
\[
\rho_{V,p',Q}(x)=\left(\frac{1}{|Q|}\int_{Q}\left|V^{\frac{1}{p'}}(t)x\right|^{p'}dt\right)^{\frac{1}{p'}}=\left(\frac{1}{|Q|}\int_{Q}\left|W^{-\frac{1}{p}}(t)x\right|^{p'}dt\right)^{\frac{1}{p'}}=\rho_{W,p',Q}^{*}(x)
\]
This yields that we can take $\mathcal{V}_{Q,p'}=\mathcal{W}'_{Q,p}$.
Analogously
\[
\rho_{V,p,Q}^{*}(x)=\left(\frac{1}{|Q|}\int_{Q}\left|V^{-\frac{1}{p'}}(t)x\right|^{p}dt\right)^{\frac{1}{p}}=\left(\frac{1}{|Q|}\int_{Q}\left|W^{\frac{1}{p}}(t)x\right|^{p}dt\right)^{\frac{1}{p}}=\rho_{W,p,Q}(x)
\]
and we can choose $\mathcal{V}'_{Q,p'}=\mathcal{W}_{Q,p}$. Consequently
we have that
\[
[V]_{A_{p'}}\simeq\|\mathcal{V}{}_{Q,p'}\mathcal{V}'_{Q,p'}\|^{p'}=\|\mathcal{W}_{Q,p}\mathcal{W}'_{Q,p}\|^{p'}
\]

The preceding discussion can be summarized in the following proposition.
\begin{prop}
\label{Prop:AqAqprime}Let $1<p<\infty$. Then
\[
[W]_{A_{p}}\simeq\left[W^{-\frac{1}{p-1}}\right]_{A_{p'}}^{\frac{1}{p-1}}.
\]
\end{prop}
In our next result we show that the $A_{1}$ type conditions constants
control the corresponding $A_{\infty}$ constants. We include in the
statement the case of the $A_{q}$ constant that was already established
in \cite{CUIM} for the sake of completeness.

\begin{prop}
If $1\leq q<\infty$ and $W\in A_{q}$, then $[W]_{A_{q,\infty}^{sc}}\leq c_{n}[W]_{A_{q}}$.
\end{prop}
\begin{proof}
We just settle the case {  $q=1$}. We observe that for every cube $Q$, { 
 a.e $y\in Q$, and every $\vec{e} \in \mathbb{C}^n$}
\[
\frac{1}{|Q|}\int_{Q}\left\Vert W^{-1}(y)W(x)\right\Vert dx=\frac{1}{|Q|}\int_{Q}{  \sup_{\vec{f}\not=0}\frac{|W(x)W^{-1}(y)\vec{f}|}{|\vec{f}|}dx\geq\frac{1}{|Q|}\int_{Q}\frac{|W(x)\vec{e}|}{|W(y)\vec{e}|}dx}
\]
or equivalently
\[
[W]_{A_{1}}|W(y){  \vec{e}}|\geq\frac{1}{|Q|}\int_{Q}|W(x){  \vec{e}}|dx.
\]
Hence
\[
[W]_{A_{1}}|W(y){  \vec{e}}|\geq\sup_{z\in Q}M(\chi_{Q}|W{  \vec{e}}|)(z)
\]
and integrating in $y$ over $Q$,
\[
\int_{Q}M(\chi_{Q}|W{  \vec{e}}|)(y)dy\leq|Q|\sup_{z\in Q}M(\chi_{Q}|W{  \vec{e}}|)(z)\leq[W]_{A_{1}}\int_{Q}|W(y){  \vec{e}}|dy
\]
Consequently
\[
\frac{1}{\int_{Q}|W(y){  \vec{e}}|dy}\int_{Q}M(\chi_{Q}|W{  \vec{e}}|)(y)dy\leq[W]_{A_{1}}
\]
and since the preceding estimate holds for every cube $Q$ and every
$e$ we have that $[W]_{A_{1,\infty}^{sc}}\leq[W]_{A_{1}}$.
\end{proof}
Now we recall the quantitative version of the reverse Hölder inequality.
This estimate was obtained for first in \cite{HPAinfty} (see \cite{HPR}
for another proof).
\begin{lem}
\label{Lem:RHI}Let $w\in A_{\infty}$ then if $0<\delta\leq\frac{1}{2^{d+11}[w]_{A_{\infty}}}$
for every cube $Q\subseteq\mathbb{R}^{d}$ we have that
\[
\left(\frac{1}{|Q|}\int_{Q}w^{1+\delta}\right)^{\frac{1}{1+\delta}}\leq\frac{2}{|Q|}\int_{Q}w.
\]
\end{lem}
We would like to end up the section presenting a technical result
that will be crucial for the proof of the main results.
\begin{lem}
\label{Lem:Key}Let $1\leq q<p<\infty$. Assume that $W\in A_{q}$
and let $r=1+\frac{1}{2^{d+11}[W]_{A_{q,\infty}^{sc}}}$. Then we
have that a.e $y\in Q$,
\[
\left(\frac{1}{|Q|}\int_{Q}\|W^{\frac{1}{p}}(x)W^{-\frac{1}{p}}(y)\|^{rp}dx\right)^{\frac{1}{rp}}\leq c_{n,p,q}\left(\frac{1}{|Q|}\int_{Q}\|W^{\frac{1}{q}}(x)W^{-\frac{1}{q}}(y)\|^{q}dx\right)^{\frac{1}{p}}.
\]
\end{lem}
\begin{proof}
Choosing ${  \vec{e}_j}(y)$ an orthonormal basis of eigenvalues $\lambda_{j}(y)$
of $W(y)$, we have by the {  classical H\"{o}lder-McCarthy inequality} (see \cite[Lemma 2.1]{B}) that
\[
\begin{split}\|W^{\frac{1}{p}}(x)W^{-\frac{1}{p}}(y)\| & \lesssim\sum_{j=1}^{n}\left|W^{\frac{1}{p}}(x)W^{-\frac{1}{p}}(y){  \vec{e}_j}(y)\right|=\sum_{j=1}^{n}\lambda_{j}(y)^{-\frac{1}{p}}\left|W^{\frac{1}{p}}(x){  \vec{e}_j}(y)\right|\\
 & \leq\sum_{j=1}^{n}\lambda_{j}(y)^{-\frac{1}{p}}\left|W^{\frac{1}{q}}(x){  \vec{e}_j}(y)\right|^{\frac{q}{p}}=\sum_{j=1}^{n}\left|W^{\frac{1}{q}}(x)\lambda_{j}(y)^{-\frac{1}{q}}{  \vec{e}_j}(y)\right|^{\frac{q}{p}}\\
 & =\sum_{j=1}^{n}\left|W^{\frac{1}{q}}(x)W^{-\frac{1}{q}}(y){  \vec{e}_j}(y)\right|^{\frac{q}{p}}\lesssim\left\Vert W^{\frac{1}{q}}(x)W^{-\frac{1}{q}}(y)\right\Vert ^{\frac{q}{p}}
\end{split}
\]
Hence
\[
\left(\frac{1}{|Q|}\int_{Q}\|W^{\frac{1}{p}}(x)W^{-\frac{1}{p}}(y)\|^{rp}dx\right)^{\frac{1}{rp}}\lesssim\left(\frac{1}{|Q|}\int_{Q}\|W^{\frac{1}{q}}(x)W^{-\frac{1}{q}}(y)\|^{qr}dx\right)^{\frac{1}{rp}}
\]
\end{proof}
Now, since $r=1+\frac{1}{2^{d+11}[W]_{A_{q,\infty}^{sc}}}$, taking
account that $W\in A_{q}\subset A_{q,\infty}^{sc}$, by reverse Hölder
inequality we have that {  choosing} any basis $\{{  \vec{e}_i}
\}_{i=1}^{n}$ of
${  \mathbb{C}^{n}}$,

\[
\begin{split}\left(\frac{1}{|Q|}\int_{Q}\|W^{\frac{1}{q}}(x)W^{-\frac{1}{q}}(y)\|^{qr}dx\right)^{\frac{1}{rp}} & \lesssim\sum_{j=1}^{n}\left(\frac{1}{|Q|}\int_{Q}\left|W^{\frac{1}{q}}(x)W^{-\frac{1}{q}}(y){  \vec{e}_i}
\right|^{qr}dx\right)^{\frac{1}{rp}}\\
 & \leq\sum_{j=1}^{n}\left(\frac{1}{|Q|}\int_{Q}\left|W^{\frac{1}{q}}(x)W^{-\frac{1}{q}}(y){  \vec{e}_i}
\right|^{q}dx\right)^{\frac{1}{p}}\\
 & \lesssim\left(\frac{1}{|Q|}\int_{Q}\|W^{\frac{1}{q}}(x)W^{-\frac{1}{q}}(y)\|^{q}dx\right)^{\frac{1}{p}}
\end{split}
\]
and we are done.

\section{Proofs of $A_{1}$ and $A_{q}$ estimates\label{sec:Proofs of A1Estimates}}

\subsection{Proof of Theorems \ref{Thm:A1Mat} and \ref{Thm:AqMat} for $M'_{W,p}$
and $M_{W,p}$}

\subsubsection{Estimates for $M'_{W,p}$}

First we deal with \eqref{eq:A1MprimaW} and \eqref{eq:AqMprimaW}. We proceed
as follows. Notice that
\[
\begin{split}\frac{1}{|Q|}\int_{Q}\left|\mathcal{W}_{Q,p}W^{-\frac{1}{p}}(y)\vec{f}(y)\right|dy & \leq\frac{1}{|Q|}\int_{Q}\left\Vert \mathcal{W}_{Q,p}W^{-\frac{1}{p}}(y)\right\Vert \left|\vec{f}(y)\right|dy\\
 & {  \lesssim} \frac{1}{|Q|}\int_{Q}\left(\frac{1}{|Q|}\int_{Q}\left\Vert W^{\frac{1}{p}}(z)W^{-\frac{1}{p}}(y)\right\Vert ^{p}dz\right)^{\frac{1}{p}}\left|\vec{f}(y)\right|dy.
\end{split}
\]
If $q=1$ it suffices to use Lemma \ref{Lem:Key} and the definition
of $A_{1}$ weight to see that
\[
\left(\frac{1}{|Q|}\int_{Q}\left\Vert W^{\frac{1}{p}}(z)W^{-\frac{1}{p}}(y)\right\Vert ^{p}dz\right)^{\frac{1}{p}}\leq c_{n,d}[W]_{A_{1}}^{\frac{1}{p}}.
\]
In that case
\[
\frac{1}{|Q|}\int_{Q}\left|\mathcal{W}_{Q,p}W^{-\frac{1}{p}}(y)\vec{f}(y)\right|dy\leq c_{n,d}[W]_{A_{1}}^{\frac{1}{p}}\frac{1}{|Q|}\int_{Q}|\vec{f}(y)|dy
\]
and using the strong type $(p,p)$ for the scalar maximal function
we are done.

If $q>1$,
\[
\begin{split} & \frac{1}{|Q|}\int_{Q}\left(\frac{1}{|Q|}\int_{Q}\left\Vert W^{\frac{1}{p}}(z)W^{-\frac{1}{p}}(y)\right\Vert ^{p}dz\right)^{\frac{1}{p}}\left|\vec{f}(y)\right|dy\\
 & \leq\left(\frac{1}{|Q|}\int_{Q}\left(\frac{1}{|Q|}\int_{Q}\left\Vert W^{\frac{1}{p}}(z)W^{-\frac{1}{p}}(y)\right\Vert ^{p}dz\right)^{\frac{q'}{p}}\right)^{\frac{1}{q'}}\left(\frac{1}{|Q|}\int_{Q}\left|\vec{f}(y)\right|^{q}dy\right)^{\frac{1}{q}}.
\end{split}
\]
Now we notice that taking into account Lemma \ref{Lem:Key},
\[
{  \left(\frac{1}{|Q|}\int_{Q}\left(\frac{1}{|Q|}\int_{Q} \norm{W^{\frac{1}{p}}(z)W^{-\frac{1}{p}}(y)}^{p}dz\right)^{\frac{q'}{p}}\right)^{\frac{1}{q'}}}\leq{  \left(\frac{1}{|Q|}\int_{Q}\left(\frac{1}{|Q|}\int_{Q} \norm{W^{\frac{1}{q}}(z)W^{-\frac{1}{q}}(y)}^{q}dz\right)^{\frac{q'}{p}}\right)^{\frac{1}{q'}}}
\]
To {  end} the estimate, observe that if we call $V=W^{-\frac{1}{q-1}}$

\begin{equation}
\begin{split} & \left[\frac{1}{|Q|}\int_{Q}\left(\frac{1}{|Q|}\int_{Q}\|W^{\frac{1}{q}}(x)W^{-\frac{1}{q}}(y)\|^{q}dx\right)^{\frac{1}{p}q'}dy\right]^{\frac{1}{q'}}\\
 & =\left[\frac{1}{|Q|}\int_{Q}\left(\frac{1}{|Q|}\int_{Q}\|V^{-\frac{1}{q'}}(x)V^{\frac{1}{q'}}(y)\|^{q}dx\right)^{\frac{q}{p}\frac{q'}{q}}dy\right]^{\frac{1}{q'}}\\
 & \leq\left[\frac{1}{|Q|}\int_{Q}\left(\frac{1}{|Q|}\int_{Q}\|V^{-\frac{1}{q'}}(x)V^{\frac{1}{q'}}(y)\|^{q}dx\right)^{\frac{q'}{q}}dy\right]^{\frac{q}{q'p}}\\
 & =\left(\left[V\right]_{A_{q'}}^{q-1}\right)^{\frac{1}{p}}\simeq\left[W\right]_{A_{q}}^{\frac{1}{p}}.
\end{split}
\label{eq:Aq}
\end{equation}
where the last step is a direct application of Proposition \ref{Prop:AqAqprime}.
Then
\[
\frac{1}{|Q|}\int_{Q}\left|\mathcal{W}_{Q,p}W^{-\frac{1}{p}}(y)\vec{f}(y)\right|\leq c_{n,d}[W]_{A_{q}}^{\frac{1}{p}}\left(\frac{1}{|Q|}\int_{Q}\left|\vec{f}(y)\right|^{q}dy\right)^{\frac{1}{q}}
\]
and using the strong type $(p,p)$ for the scalar operator $M_{q}(f)=M(|f|^{q})^{\frac{1}{q}}$,
we are done.

\subsubsection{Estimates for $M_{W,p}$}

We are going to settle \eqref{eq:A1MW} and \eqref{eq:AqMW} {  at}  the same time. First we note that {  by the proof of Lemma \ref{Lem:Key}}
 we have that
\begin{equation}
\|\MC{W}_{Q,q}^{\frac{q}{p}}W^{-\frac{1}{p}}(y)\|\lesssim\|\MC{W}_{Q,q}W^{-\frac{1}{q}}(y)\|^{\frac{q}{p}}\label{HM}
\end{equation}
{  for any $Q \in \D$.  Fix $J \in \D$.}  Let $\mathcal{J}(J)$ denote the maximal cubes $L\in\mathcal{D}(J)$
(if any exist) where
\begin{equation}
\innp{|\MC{W}_{J,q}^{\frac{q}{p}}W^{-\frac{1}{p}}\V{f}|}_{L}>4\innp{|\MC{W}_{J,q}^{\frac{q}{p}}W^{-\frac{1}{p}}\V{f}|}_{J}. \label{StopCrit}
\end{equation}
By maximality, {  as usual,}  we have
\begin{align*}
\sum_{L\in\mathcal{J}(J)}|L| & \leq\frac{1}{2\innp{|\MC{W}_{J,q}^{\frac{q}{p}}W^{-\frac{1}{p}}\V{f}|}_{J}}\sum_{L\in\mathcal{J}(J)}\int_{L}|\MC{W}_{J,q}^{\frac{q}{p}}W^{-\frac{1}{p}}(y)\V{f}(y)|\,dy\\
 & \leq\frac{1}{2\innp{|\MC{W}_{J,q}^{\frac{q}{p}}W^{-\frac{1}{p}}\V{f}|}_{J}}\int_{J}|\MC{W}_{J,q}^{\frac{q}{p}}W^{-\frac{1}{p}}(y)\V{f}(y)|\,dy\\
 & =\frac{|J|}{{ 4}}.
\end{align*}

Now let $\MC{F}(J)$ be the collection of cubes in $\mathcal{D}(J)$
that are not a subset of any cube $I\in\mathcal{J}(J)$. Furthermore,
for ease of notation let $\cup\mathcal{J}(J)=\cup_{L\in\mathcal{J}(J)}L$.
Let
\begin{align*}
M_{J,W}\V{f}(x)=\sup_{\substack{Q\ni x\\
Q\in\mathcal{D}(J)
}
}\dashint_{Q}|W^{\frac{1}{p}}(x)W^{-\frac{1}{p}}(y)\V{f}(y)|\,dy
\end{align*}

We pointwise dominate $M_{J,W}\V{f}(x)$ by looking at three cases.
First, assume $Q\in\mathcal{F}(J)$ and assume $x\in\cup\mathcal{J}(J)$.
Thus, let $x\in Q\in\mathcal{F}(J)$ and $x\in I\in\mathcal{J}(J)$.
Then by definition of $\mathcal{F}(J)$ we must have $I\subsetneq Q\subseteq J$
so that in this case,  \eqref{HM} and \eqref{StopCrit} gives us

\begin{align*}
 & \dashint_{Q}|W^{\frac{1}{p}}(x)W^{-\frac{1}{p}}(y)\V{f}(y)|\,dy\leq\sup_{I\in\mathcal{J}(J)}\sup_{J\supseteq Q\varsupsetneq I\ni x}\dashint_{Q}|W^{\frac{1}{p}}(x)W^{-\frac{1}{p}}(y)\V{f}(y)|\,dy\\
 & \leq\|W^{\frac{1}{p}}(x)\W_{J, q}^{-\frac{q}{p}}\|\sup_{I\in\mathcal{J}(J)}\sup_{J\supseteq Q\varsupsetneq I}\dashint_{Q}|{ { \W_{J,q}}}^{\frac{q}{p}}W^{-\frac{1}{p}}(y)\V{f}(y)|\,dy\\
 & \leq{ 4}\|W^{\frac{1}{p}}(x){ \W_{J, q}}^{-\frac{q}{p}}\|\dashint_{J}|\W_{J,q}^{\frac{q}{p}}W^{-\frac{1}{p}}(y)\V{f}(y)|\,dy\\
 & \leq{ 4}\|W^{\frac{1}{p}}(x){ \W_{J, q}}^{-\frac{q}{p}}\|\dashint_{J}\|\W_{J,q}^{\frac{q}{p}}W^{-\frac{1}{p}}(y)\||\V{f}(y)|\,dy\\
 & \leq{ 4}\|W^{\frac{1}{q}}(x){ \W_{J, q}}^{-1}\|^{\frac{q}{p}}\dashint_{J}\|\W_{J,q}W^{-\frac{1}{q}}(y)\|^{\frac{q}{p}}|\V{f}(y)|\,dy\\
 & =A
\end{align*}
at this point if $q=1$ we have that
\[
\begin{split}A & \leq { 4}\|W(x)\W_{J,1}^{-1}\|^{\frac{1}{p}}\dashint_{J}\|\W_{J,1}W^{-1}(y)\|^{\frac{1}{p}}|\V{f}(y)|\,dy\\
 & \leq{  4 c_n} [W]_{\text{A}_{1}}^{\frac{1}{p}}\|W(x)\W_{J,1}^{-1}\|^{\frac{1}{p}}\innp{|\V{f}|}_{J}
\end{split}
\]
and if $1<q<\infty$,
\[
\begin{split}A & \leq{  4} c_{n}\|W^{\frac{1}{q}}(x)\W_{J,q}^{-1}\|^{\frac{q}{p}}\dashint_{J}\left(\dashint_{J}\|W^{{ \frac{1}{q}}}(x)W^{-\frac{1}{q}}(y)\|^{q}dx\right)^{\frac{1}{q}\frac{q}{p}}|\V{f}(y)|\,dy\\
 & \leq{ 4}c_{n}\|W^{\frac{1}{q}}(x)\W_{J,q}^{-1}\|^{\frac{q}{p}}\left(\dashint_{J}\left(\dashint_{J}\|W^{ \frac{1}{q}}(x)W^{-\frac{1}{q}}(y)\|^{q}dx\right)^{\frac{q'}{p}}dy\right)^{\frac{1}{q'}}\left(\dashint_{J}|\V{f}(y)|^{q}\,dy\right)^{\frac{1}{q}}\\
 & \leq{ 4}c_{n}\|W^{\frac{1}{q}}(x)\W_{J,q}^{-1}\|^{\frac{q}{p}}[W]_{A_{q}}^{\frac{1}{p}}\left(\dashint_{J}|\V{f}(y)|^{q}\,dy\right)^{\frac{1}{q}}.
\end{split}
\]
Next, assume $Q\in\mathcal{F}(J)$ and $x\not\in\cup\mathcal{J}(J)$.
Pick a sequence $L_{k}^{x}$ of nested dyadic cubes where
\[
\{L_{k}^{x}\}=\{L\in\mathcal{F}(J):x\in L\}=\{L\in\mathcal{D}(J):x\in L\}.
\]

But if
\[
\sup_{k}\innp{|\MC{W}_{J,q}^{\frac{q}{p}}W^{-\frac{1}{p}}\V{f}|}_{L_{k}^{x}}>{ 4}\innp{|\MC{W}_{J,q}^{\frac{q}{p}}W^{-\frac{1}{p}}\V{f}|}_{J}
\]
then for some $k$ we have
\[
\innp{|\MC{W}_{J,q}^{\frac{q}{p}}W^{-\frac{1}{p}}\V{f}|}_{L_{k}^{x}}>{ 4}\innp{|\MC{W}_{J,q}^{\frac{q}{p}}W^{-\frac{1}{p}}\V{f}|}_{J}
\]
which means that $x\in L_{k}^{x}\subseteq Q$ for some $Q\in\mathcal{J}(J)$.
Thus, bearing the computation above in mind,
\begin{align*}
\dashint_{Q}|W^{\frac{1}{p}}(x)W^{-\frac{1}{p}}(y)\V{f}(y)|\,dy & \leq\sup_{k}\dashint_{L_{k}^{x}}|W^{\frac{1}{p}}(x)W^{-\frac{1}{p}}(y)\V{f}(y)|\,dy\\
 & \leq\|W^{\frac{1}{p}}(x)\W_{J,q}^{-\frac{q}{p}}\|\sup_{k}\innp{|\MC{W}_{J,q}^{\frac{q}{p}}W^{-\frac{1}{p}}\V{f}|}_{L_{k}^{x}}\\
 & \leq{ 4}\|W^{\frac{1}{p}}(x)\W_{J,q}^{-\frac{q}{p}}\|\innp{|\MC{W}_{J,q}^{\frac{q}{p}}W^{-\frac{1}{p}}\V{f}|}_{J}\\
 & \leq{ 4}c_{n}\|W^{\frac{1}{q}}(x)\W_{J,q}^{-1}\|^{\frac{q}{p}}[W]_{A_{q}}^{\frac{1}{p}}\left(\dashint_{J}|\V{f}(y)|^{q}\,dy\right)^{\frac{1}{q}}
\end{align*}
in the case $1<q<\infty$ and
\[
\dashint_{Q}|W^{\frac{1}{p}}(x)W^{-\frac{1}{p}}(y)\V{f}(y)|\,dy\lesssim[W]_{\text{A}_{1}}^{\frac{1}{p}}\|W(x)\W_{J,1}^{-1}\|^{\frac{1}{p}}\innp{|\V{f}|}_{J}
\]
in the case $q=1$.

Lastly, if $Q\not\in\mathcal{F}(J)$ then $Q\subseteq L$ for some
$L\in\mathcal{J}(J)$ so obviously if $x\in Q$ then $x\in\cup\mathcal{J}(J)$.
Combining all this gives
\[
M_{J,W}\V{f}(x)\leq\max\left\{ { 4c_{n} \chi_{J \backslash \bigcup\MC{J}(J)} (x)}\|W^{\frac{1}{q}}(x)\W_{J,q}^{-1}\|^{\frac{q}{p}}[W]_{A_{q}}^{\frac{1}{p}}\left(\dashint_{J}|\V{f}(y)|^{q}\,dy\right)^{\frac{1}{q}},\ \chi_{\cup\mathcal{J}(J)}(x)\sup_{L\in\mathcal{J}(J)}M_{L,W}\V{f}(x)\right\}
\]
Thus, for $1\leq q<\infty,$
\[
\begin{split} & \int_{J}(M_{J,W}\V{f}(x))^{p}\,dx=({ 4}c_{n})^{p}[W]_{\text{A}_{q}}\innp{|\V{f}|}_{J,q}^{p}\int_{J}\|W^{\frac{1}{q}}(x)\W_{J,q}^{-1}\|^{{ {q}}}+\sum_{L\in\mathcal{J}(J)}\int_{L}(M_{L,W}\V{f}(x))^{p}\,dx\\
 & =({ 4}c_{n})^{p}[W]_{\text{A}_{q}}\innp{|\V{f}|}_{J,q}^{p}|J|{ \left(\dashint_{J}\|W^{\frac{1}{q}}(x)\W_{J,q}^{-1}\|^{q}\right)}+\sum_{L\in\mathcal{J}(J)}\int_{L}(M_{L,W}\V{f}(x))^{p}\,dx\\
 & \lesssim({ 4}c_{n})^{p}[W]_{\text{A}_{q}}\innp{|\V{f}|}_{J,q}^{p}|J|+\sum_{L\in\mathcal{J}(J)}\int_{L}(M_{L,W}\V{f}(x))^{p}\,dx
\end{split}
\]
If as usual $\mathcal{J}_{k}(J)=\{L\in\mathcal{J}(Q):Q\in\mathcal{J}_{k-1}(J)\}$
with $\mathcal{J}_{0}(J)=\{J\}$ and $\mathcal{S}=\cup_{k}\mathcal{J}_{k}(J)$
then $\mathcal{S}$ is sparse and iteration gives us
\begin{align*}
\int_{J}(M_{J,W}\V{f}(x))^{p}\,dx & \lesssim[W]_{\text{A}_{q}}\sum_{L\in\mathcal{S}}|L|\inf_{x\in L}(M_{q}(|\V{f}|)(x))^{p}\\
 & \lesssim[W]_{\text{A}_{q}}\sum_{L\in\mathcal{S}}\int_{L}(M_{q}(|\V{f}|)(x))^{p}\,dx\\
 & \lesssim[W]_{\text{A}_{q}}\sum_{L\in\mathcal{S}}\int_{E_{L}}(M_{q}(|\V{f}|)(x))^{p}\,dx\\
 & \lesssim[W]_{\text{A}_{q}}\|M_{q}(|\V{f}|)\|_{L^{p}}^{p}\\
 & \lesssim[W]_{\text{A}_{q}}\|\V{f}\|_{L^{p}}^{p}.
\end{align*}

\subsection{Proof of Theorems \ref{Thm:A1Mat} and \ref{Thm:AqMat} for singular
operators}

\subsubsection{A reduction to bump conditions}

We recall that $A:[0,\infty)\rightarrow[0,\infty)$ is Young function
if $A(0)=0$ and it is a convex and increasing function. Given a function
$f$ and a measurable set $E$ with finite measure, we can define
the average on $E$ of $f$ associated to $A$ by
\[
\|f\|_{A,E}=\inf\left\{ \lambda>0\,:\,\frac{1}{|E|}\int_{E}A\left(\frac{|f|}{\lambda}\right)\leq1\right\} .
\]
From that definition it readily follows that if
\[
\lambda_{1}\leq\|f\|_{A,Q}\leq\lambda_{2}
\]
then
\begin{equation}
\frac{1}{|Q|}\int_{Q}A\left(\frac{|f|}{\lambda_{2}}\right)\leq1\quad\text{and}\quad\frac{1}{|Q|}\int_{Q}A\left(\frac{|f|}{\lambda_{1}}\right)\geq1.\label{eq:Cont}
\end{equation}
Given a Young function $A$ it is natural to define a maximal operator
$M_{A}$ hinging upon the preceding definition of average as follows
\[
M_{A}f(x)=\sup_{Q\ni x}\|f\|_{A,Q}.
\]
The boundedness of those operators on $L^{p}$ spaces was thoroughly
studied by C. Pérez \cite{PMax}, under the aditional condition that
$A$ is doubling, assumption that was proved to be superfluous by
Liu and Luque \cite{LL}. The condition is the following
\begin{equation}
\|M_{A}\|_{L^{p}\rightarrow L^{p}}\leq {  c_{d}}\left(\int_{1}^{\infty}\frac{A(t)}{t^{p}}\frac{dt}{t}\right)^{\frac{1}{p}}\label{eq:ConstMA}
\end{equation}
Associated to each Young function we can define the so called associated
Young function $\overline{A}$ by
\[
\overline{A}(t)=\sup_{s>0}\{st-A(s)\}.
\]
That function has some interesting properties. The first of them is
that
\[
t\leq A^{-1}(t)\overline{A}^{-1}(t)\leq2t.
\]
The second one, that will be very interesting for us, is the following
generalized Hölder inequality
\[
\frac{1}{|Q|}\int_{Q}|fg|\leq2\|f\|_{A,Q}\|g\|_{\overline{A},Q}.
\]
For more details about Young functions we remit the reader to \cite{O,RR}.

Let $T$ is a Calderón-Zygmund operator and $W,V$ be matrix weights. If we call

\[
T_{\mathcal{S}}^{W,V}\phi(x)=\sum_{Q\in\mathcal{S}}\frac{1}{|Q|}\int_{Q}\left\Vert W^{\frac{1}{p}}(x)V^{-\frac{1}{p}}(y)\right\Vert \phi(y)dy\chi_{Q}(x)
\]
and \[
\begin{split}[b,T]_{\mathcal{S}}^{W,V}\phi(x) & =\sum_{Q\in\mathcal{S}}\frac{1}{|Q|}|b(x)-\langle b\rangle_{Q}|\int_{Q}\left\Vert W^{\frac{1}{p}}(x)V^{-\frac{1}{p}}(y)\right\Vert \phi(y)dy\chi_{Q}(x)\\
 & +\sum_{Q\in\mathcal{S}}\frac{1}{|Q|}\int_{Q}|b(y)-\langle b\rangle_{Q}|\left\Vert W^{\frac{1}{p}}(x)V^{-\frac{1}{p}}(y)\right\Vert \phi(y)dy\chi_{Q}(x)
\end{split}
\]
then Theorems \ref{Thm:ConvexBody} and  \ref{Thm:ConvexBodyComm} immediately give us that
\[
\|T\|_{L^{p}(V)\rightarrow L^{p}(W)}\lesssim\sup_{\mathcal{S}}\|T_{\mathcal{S}}^{W,V}\|_{L^{p}(\mathbb{R}^{d})\rightarrow L^{p}(\mathbb{R}^{d})}
\]
 and \[
\|[b,T]\|_{L^{p}(V)\rightarrow L^{p}(W)}\lesssim\sup_{\mathcal{S}}\|[b,T]_{\mathcal{S}}^{W,V}\|_{L^{p}(\mathbb{R}^{d})\rightarrow L^{p}(\mathbb{R}^{d})}.
\]

Armed with the preceding definitions and results and arguing in the
spirit of \cite{CUIM} we can prove a lemma that will be fundamental
for our purposes.
\begin{lem}
\label{Lem:Bumps}Let $A,B$ be Young functions. Then
\[
\|T_{\mathcal{S}}^{W,V}\|_{{  L^{p}(\mathbb{R}^{d}; \Cn)\rightarrow L^{p}(\mathbb{R}^{d}; \Cn)}} { \lesssim}\|M_{\overline{A}}\|_{L^{p'}}\|M_{\overline{B}}\|_{L^{p}}\min\{\kappa_{1},\kappa_{2}\}
\]
where $\kappa_{1}=\sup_{Q}\|\|\|W^{\frac{1}{p}}(x)V^{-\frac{1}{p}}(y)\|\|_{A_{x},Q}\|_{B_{y},Q}$
and $\kappa_{2}=\sup_{Q}\|\|\|W^{\frac{1}{p}}(x)V^{-\frac{1}{p}}(y)\|\|_{B_{y},Q}\|_{A_{x},Q}$.
\end{lem}
\begin{proof}
Without loss of generality we may assume that $f,g\geq0$. Then taking
into account generalized Hölder inequality

\begingroup
\allowdisplaybreaks

$\begin{aligned} & \sum_{Q\in\mathcal{S}}\frac{1}{|Q|}\int_{Q}\int_{Q}\|W^{\frac{1}{p}}(x)V^{-\frac{1}{p}}(y)\|f(y)g(x)dydx\\
 & =\sum_{Q\in\mathcal{S}}\frac{1}{|Q|}\int_{Q}\int_{Q}\|W^{\frac{1}{p}}(x)V^{-\frac{1}{p}}(y)\|g(x)f(y)dxdy\\
 & \leq2\sum_{Q\in\mathcal{S}}\|g\|_{\overline{A},Q}\int_{Q}\|\|W^{\frac{1}{p}}(x)V^{-\frac{1}{p}}(y)\|\|_{A_{x},Q}f(y)dy\leq\\
 & \leq c\sum_{Q\in\mathcal{S}}\|f\|_{\overline{B},Q}\|g\|_{\overline{A},Q}|E_{Q}|\|\|\|W^{\frac{1}{p}}(x)V^{-\frac{1}{p}}(y)\|\|_{A_{x},Q}\|_{B_{y},Q}\\
 & \leq c\sup_{S}\|\|\|W^{\frac{1}{p}}(x)V^{-\frac{1}{p}}(y)\|\|_{A_{x},Q}\|_{B_{y},S}\sum_{Q\in\mathcal{S}}\|f\|_{\overline{B},Q}\|g\|_{\overline{A},Q}|E_{Q}|\\
 & \leq c\|M_{\overline{A}}\|_{L^{p'}}\|M_{\overline{B}}\|_{L^{p}}\sup_{Q}\|\|\|W^{\frac{1}{p}}(x)V^{-\frac{1}{p}}(y)\|\|_{A_{x},Q}\|_{B_{y},Q}\|f\|_{ { L^{p}(\mathbb{R}^{d})}}\|g\|_{ { L^{p'}(\mathbb{R}^{d})}}.
\end{aligned}
$
\endgroup

The other estimate is obtained arguing analogously.
\end{proof}
In the case of commutators we can provide the following counterpart
\begin{lem}
\label{Lem:BumpsComm}Let $A,B,C,D$ be Young functions. Then
\[
\|[b,T]_{\mathcal{S}}^{W,V}\|_{{  L^{p}(\mathbb{R}^{d}; \Cn)\rightarrow L^{p}(\mathbb{R}^{d}; \Cn)}}{ \lesssim}(\Lambda_{1}+\Lambda_{2})
\]
where $\Lambda_{1}=\|M_{\overline{A}}\|_{L^{p'}}\|M_{\overline{B}}\|_{L^{p}}\min\left\{ \kappa_{1},\kappa_{2}\right\} $
with
\[
\begin{split}\kappa_{1} & =\sup_{Q}\|\||b(x)-\langle b\rangle_{Q}|\|W^{\frac{1}{p}}(x)V^{-\frac{1}{p}}(y)\|\|_{A_{x},Q}\|_{B_{y},Q}\\
\kappa_{2} & =\sup_{S}\||b(x)-\langle b\rangle_{Q}|\|\|W^{\frac{1}{p}}(x)V^{-\frac{1}{p}}(y)\|\|_{B_{y},Q}\|_{A_{x},Q}
\end{split}
\]
 and $\Lambda_{2}=\|M_{\overline{C}}\|_{L^{p'}}\|M_{\overline{D}}\|_{L^{p}}\min\left\{ \kappa_{3},\kappa_{4}\right\} $
with
\[
\begin{split}\kappa_{3} & =\sup_{Q}\||b(y)-\langle b\rangle_{Q}|\|\|W^{\frac{1}{p}}(x)V^{-\frac{1}{p}}(y)\|\|_{C_{x},Q}\|_{D_{y},Q}\\
\kappa_{4} & =\sup_{Q}\|\||b(y)-\langle b\rangle_{Q}|\|W^{\frac{1}{p}}(x)V^{-\frac{1}{p}}(y)\|\|_{D_{y},Q}\|_{C_{x},Q}.
\end{split}
\]
\end{lem}
\begin{proof}
We recall that
\[
\begin{split}[b,T]_{\mathcal{S}}^{W,V}\phi(x) & =\sum_{Q\in\mathcal{S}}\frac{1}{|Q|}|b(x)-\langle b\rangle_{Q}|\int_{Q}\left\Vert W^{\frac{1}{p}}(x)V^{-\frac{1}{p}}(y)\right\Vert \phi(y)dy\chi_{Q}(x)\\
 & +\sum_{Q\in\mathcal{S}}\frac{1}{|Q|}\int_{Q}|b(y)-\langle b\rangle_{Q}|\left\Vert W^{\frac{1}{p}}(x)V^{-\frac{1}{p}}(y)\right\Vert \phi(y)dy\chi_{Q}(x)
\end{split}
\]
Without loss of generality we may assume that $f,g\geq0$. For the
first term we can argue as follows

\begingroup
\allowdisplaybreaks

$\begin{aligned} & \sum_{Q\in\mathcal{S}}\frac{1}{|Q|}\int_{Q}|b(x)-\langle b\rangle_{Q}|\int_{Q}\|W^{\frac{1}{p}}(x)V^{-\frac{1}{p}}(y)\|f(y)g(x)dydx\\
 & =\sum_{Q\in\mathcal{S}}\frac{1}{|Q|}\int_{Q}\int_{Q}|b(x)-\langle b\rangle_{Q}|\|W^{\frac{1}{p}}(x)V^{-\frac{1}{p}}(y)\|g(x)f(y)dxdy\\
 & \leq2\sum_{Q\in\mathcal{S}}\|g\|_{\overline{A},Q}\int_{Q}\||b(x)-\langle b\rangle_{Q}|\|W^{\frac{1}{p}}(x)V^{-\frac{1}{p}}(y)\|\|_{A_{x},Q}f(y)dy\leq\\
 & \leq c\sum_{Q\in\mathcal{S}}\|f\|_{\overline{B},Q}\|g\|_{\overline{A},Q}|E_{Q}|\|\||b(x)-\langle b\rangle_{Q}|\|W^{\frac{1}{p}}(x)V^{-\frac{1}{p}}(y)\|\|_{A_{x},Q}\|_{B_{y},Q}\\
 & \leq c\sup_{S}\|\||b(x)-b_{Q}|\|W^{\frac{1}{p}}(x)V^{-\frac{1}{p}}(y)\|\|_{A_{x},Q}\|_{B_{y},S}\sum_{Q\in\mathcal{S}}\|f\|_{\overline{B},Q}\|g\|_{\overline{A},Q}|E_{Q}|\\
 & \leq c\|M_{\overline{A}}\|_{L^{p'}}\|M_{\overline{B}}\|_{L^{p}}\sup_{Q}\|\||b(x)-\langle b\rangle_{Q}|\|W^{\frac{1}{p}}(x)V^{-\frac{1}{p}}(y)\|\|_{A_{x},Q}\|_{B_{y},Q}\|f\|_{{  L^{p}(\mathbb{R}^{d})}\|g\|_{L^{p'}(\mathbb{R}^{d})}}.
\end{aligned}
$

\endgroup

Arguing analogously we obtain the rest of the estimates.
\end{proof}

\subsubsection{Proof of estimates \eqref{eq:A1Mat} and \eqref{eq:A1MatComm}}

Again we deal first with \eqref{eq:A1Mat}. We will use Lemma \ref{Lem:Bumps}.
Let us choose $\overline{B}(t)=t^{\frac{p+1}{2}}$ and $A(t)=t^{rp}$
with $r=1+\frac{1}{2^{d+11}[W]_{A_{1,\infty}^{sc}}}$.
We observe that $B(t)\simeq t^{\frac{p+1}{p-1}}$ . It's not hard
to check that ${ \|M_{\overline{B}}\|_{ {L^{p} \rightarrow L^p}}\leq c_{d}(p')^\frac{1}{p}} $ and that
${  \|M_{\overline{A}}\|_{L^{p'} \rightarrow L^{p'}}\leq c_{d} p^\frac{1}{p'} [W]_{A_{1,\infty}^{sc}}^{\frac{1}{p'}}}$
. We observe that using Lemma \ref{Lem:Key} and the definition of
$A_{1}$ weight,
\begin{equation}
\begin{split} & \|\|\|W^{\frac{1}{p}}(x)W^{-\frac{1}{p}}(y)\|\|_{A_{x},Q}\|_{B_{y},Q}\\
 & =\left[\frac{1}{|Q|}\int_{Q}\left(\frac{1}{|Q|}\int_{Q}\|W^{\frac{1}{p}}(x)W^{-\frac{1}{p}}(y)\|^{rp}dx\right)^{\frac{1}{rp}\frac{p+1}{p-1}}dy\right]^{\frac{p-1}{p+1}}\\
 & \leq c_{n,p}\left[\frac{1}{|Q|}\int_{Q}\left(\frac{1}{|Q|}\int_{Q}\|W(x)W^{-1}(y)\|dx\right)^{\frac{1}{p}\frac{p+1}{p-1}}dy\right]^{\frac{p-1}{p+1}}\\
 & \leq c_{n,p}\left[\frac{1}{|Q|}\int_{Q}\left([W]_{A_{1}}\right)^{\frac{1}{p}\frac{p+1}{p-1}}dy\right]^{\frac{p-1}{p+\text{1}}}=c_{n,p}[W]_{A_{1}}^{\frac{1}{p}}
\end{split}
\label{eq:A1Mat1st}
\end{equation}
 and we are done.

Now we turn our attention to \eqref{eq:A1MatComm}. We use Lemma \ref{Lem:BumpsComm}.
First we choose $\overline{B}(t)=t^{\frac{p+1}{2}}$ and $A(t)=t^{sp}$
with $s=\frac{r+1}{2}$ and $r=1+\frac{1}{2^{d+11}[W]_{A_{1,\infty}^{sc}}}$.
For that choice of $s$ we have that $\left(\frac{r}{s}\right)' {  = } 2r'$.
Notice that again $B(t)\simeq t^{\frac{p+1}{p-1}}, \ {  \|M_{\overline{B}}\|_{L^{p} \rightarrow L^p}\leq c_{}(p')^\frac{1}{p}}, $ and that $\|M_{\overline{A}}\|_{L^{p'} \rightarrow L^{p'}}\leq  c_{d}p^\frac{1}{p'} [W]_{A_{1,\infty}^{sc}} ^{\frac{1}{p'}}$
. On the other hand,
\[
\begin{split} & \|\||b(x)-\langle b\rangle_{Q}|\|W^{\frac{1}{p}}(x)W^{-\frac{1}{p}}(y)\|\|_{A_{x},Q}\|_{B_{y},Q}\\
 & =\left[\frac{1}{|Q|}\int_{Q}\left(\frac{1}{|Q|}\int_{Q}|b(x)-\langle b\rangle_{Q}|^{sp}\|W^{\frac{1}{p}}(x)W^{-\frac{1}{p}}(y)\|^{sp}dx\right)^{\frac{1}{sp}\frac{p+1}{p-1}}dy\right]^{\frac{p-1}{p+1}}\\
 & \leq\left(\frac{1}{|Q|}\int_{Q}|b(x)-\langle b\rangle_{Q}|^{sp\left(\frac{r}{s}\right)'}dx\right)^{\frac{1}{sp\left(\frac{r}{s}\right)'}}\left[\frac{1}{|Q|}\int_{Q}\left(\frac{1}{|Q|}\int_{Q}\|W^{\frac{1}{p}}(x)W^{-\frac{1}{p}}(y)\|^{rp}dx\right)^{\frac{1}{rp}\frac{p+1}{p-1}}dy\right]^{\frac{p-1}{p+1}}\\
  & \leq c_{d}sp\left(\frac{r}{s}\right)'\|b\|_{BMO}\left[\frac{1}{|Q|}\int_{Q}\left(\frac{1}{|Q|}\int_{Q}\|W^{\frac{1}{p}}(x)W^{-\frac{1}{p}}(y)\|^{rp}dx\right)^{\frac{1}{rp}\frac{p+1}{p-1}}dy\right]^{\frac{p-1}{p+1}}.
\end{split}
\]
From this point arguing as in \eqref{eq:A1Mat1st} we have that
\[
\begin{split} & \|\||b(x)-\langle b\rangle_{Q}|\|W^{\frac{1}{p}}(x)W^{-\frac{1}{p}}(y)\|\|_{A_{x},Q}\|_{B_{y},Q}\\
 & \leq c_{n,d}sp\left(\frac{r}{s}\right)'\|b\|_{BMO}[W]_{A_{1}}^{\frac{1}{p}}\\
 & \leq c_{n,d,p}\|b\|_{BMO}[W]_{A_{1,\infty}^{sc}}[W]_{A_{1}}^{\frac{1}{p}}
\end{split}
.
\]
For the other term, we choose $\overline{D}(t)=t^{\frac{p+1}{2}}$
and $C(t)=t^{rp}$ with $r=1+\frac{1}{2^{d+11}[W]_{A_{1,\infty}^{sc}}}$. Then
\[
\begin{split} & \||b(y)-\langle b\rangle_{Q}|\|\|W^{\frac{1}{p}}(x)W^{-\frac{1}{p}}(y)\|\|_{C_{x},Q}\|_{D_{y},Q}\\
 & =\left[\frac{1}{|Q|}\int_{Q}|b(y)-\langle b\rangle_{Q}|^{\frac{p+1}{p-1}}\left(\frac{1}{|Q|}\int_{Q}\|W^{\frac{1}{p}}(x)W^{-\frac{1}{p}}(y)\|^{rp}dx\right)^{\frac{1}{rp}\frac{p+1}{p-1}}dy\right]^{\frac{p-1}{p+1}}
\end{split}
\]
Arguing as above,
\[
\left(\frac{1}{|Q|}\int_{Q}\|W^{\frac{1}{p}}(x)W^{-\frac{1}{p}}(y)\|^{rp}dx\right)^{\frac{1}{rp}\frac{p+1}{p-1}}\leq c_{n}{  [W]_{A_{1}}^{\frac{1}{p} \frac{p-1}{p+1}}}.
\]
Hence
\[
\begin{split} & \left[\frac{1}{|Q|}\int_{Q}|b(y)-\langle b\rangle_{Q}|^{\frac{p+1}{p-1}}\left(\frac{1}{|Q|}\int_{Q}\|W^{\frac{1}{p}}(x)W^{-\frac{1}{p}}(y)\|^{rp}dx\right)^{\frac{1}{rp}\frac{p+1}{p-1}}dy\right]^{\frac{p-1}{p+1}}\\
 & \leq c_{n}[W]_{A_{1}}^{\frac{1}{p}}\left[\frac{1}{|Q|}\int_{Q}|b(y)-\langle b\rangle_{Q}|^{\frac{p+1}{p-1}}dy\right]^{\frac{p-1}{p+1}}\leq c_{n,d,p}\|b\|_{BMO}[W]_{A_{1}}^{\frac{1}{p}}
\end{split}
\]
Consequently gathering all the preceding estimates we obtain \eqref{eq:A1MatComm}.

\subsubsection{Proof of estimates \eqref{eq:AqMat} and \eqref{eq:AqMatComm}}

We deal first with \eqref{eq:AqMat}. We rely again upon {  Lemma \ref{Lem:Bumps}.}
We note that choosing $A(t)=t^{rp}$ with $r=1+\frac{1}{2^{d+11}[W]_{A_{q,\infty}^{sc}}}$
and $B(t)=t^{q'}$ we have that ${  \|M_{\overline{A}}\|_{L^{p'} \rightarrow L^{p'}}\leq c_{d}(r')^{\frac{1}{p'}}\leq c_{d}[W]_{A_{q,\infty}^{sc}}^{\frac{1}{p'}}}$
and $\|M_{\overline{B}}\|_{L^{p}}\leq c_{d,p,q}$. On the other hand,
notice that
\[
\begin{split} & \|\|\|W^{\frac{1}{p}}(x)W^{-\frac{1}{p}}(y)\|\|_{A_{x},Q}\|_{B_{y},Q}\\
 & =\left[\frac{1}{|Q|}\int_{Q}\left(\frac{1}{|Q|}\int_{Q}\|W^{\frac{1}{p}}(x)W^{-\frac{1}{p}}(y)\|^{rp}dx\right)^{\frac{1}{rp}q'}dy\right]^{\frac{1}{q'}}.
\end{split}
\]
By Lemma \ref{Lem:Key}
\begin{equation}
\left(\frac{1}{|Q|}\int_{Q}\|W^{\frac{1}{p}}(x)W^{-\frac{1}{p}}(y)\|^{rp}dx\right)^{\frac{1}{rp}}\leq c_{n,d,p,q}\left(\frac{1}{|Q|}\int_{Q}\|W^{\frac{1}{q}}(x)W^{-\frac{1}{q}}(y)\|^{q}dx\right)^{\frac{1}{p}}.\label{eq:RHolderAq}
\end{equation}
  Then
\[
\|\|\|W^{\frac{1}{p}}(x)W^{-\frac{1}{p}}(y)\|\|_{A_{x},Q}\|_{B_{y},Q}\leq c_{n,d,p,q}\left[\frac{1}{|Q|}\int_{Q}\left(\frac{1}{|Q|}\int_{Q}\|W^{\frac{1}{q}}(x)W^{-\frac{1}{q}}(y)\|^{q}dx\right)^{\frac{q'}{p}}dy\right]^{\frac{1}{q'}}
\]
by \eqref{eq:Aq}  we have that
\[
\|\|\|W^{\frac{1}{p}}(x)W^{-\frac{1}{p}}(y)\|\|_{A_{x},Q}\|_{B_{y},Q}\leq c_{n,d,p,q}[W]_{A_{q}}^{\frac{1}{p}}.
\]
Gathering the preceding estimates and taking into account that $r=1+\frac{1}{2^{d+11}[W]_{A_{q,\infty}^{sc}}}$
we obtain that \eqref{eq:AqMat} holds.

Let us deal now with \eqref{eq:AqMatComm}. Arguing analogously as
above, we will use Lemma \ref{Lem:BumpsComm}. First we choose $B(t)=t^{q'}$
and $A(t)=t^{sp}$ with $s=\frac{r+1}{2}$ and $r=1+\frac{1}{2^{d+11}[W]_{A_{q,\infty}^{sc}}}$.
For that choice of $s$ we have that $\left(\frac{r}{s}\right)' {   =} 2r'$.
It is also straightforward that $\|M_{\overline{B}}\|_{L^{p}}\leq c_{n,p,q}$
and that $\|M_{\overline{A}}\|_{L^{p'}}\leq c_{n,d}p\left(s'\right)^{\frac{1}{p'}}\leq c_{n,d}p[W]_{A_{q,\infty}^{sc}}^{\frac{1}{p'}}$
. Then we have that
\[
\begin{split} & \|\||b(x)-\langle b\rangle_{Q}|\|W^{\frac{1}{p}}(x)W^{-\frac{1}{p}}(y)\|\|_{A_{x},Q}\|_{B_{y},Q}\\
 & =\left[\frac{1}{|Q|}\int_{Q}\left(\frac{1}{|Q|}\int_{Q}|b(x)-\langle b\rangle_{Q}|^{sp}\|W^{\frac{1}{p}}(x)W^{-\frac{1}{p}}(y)\|^{sp}dx\right)^{\frac{1}{sp}q'}dy\right]^{\frac{1}{q'}}\\
 & \leq\left(\frac{1}{|Q|}\int_{Q}|b(x)-\langle b\rangle_{Q}|^{sp\left(\frac{r}{s}\right)'}dx\right)^{\frac{1}{sp\left(\frac{r}{s}\right)'}}\left[\frac{1}{|Q|}\int_{Q}\left(\frac{1}{|Q|}\int_{Q}\|W^{\frac{1}{p}}(x)W^{-\frac{1}{p}}(y)\|^{rp}dx\right)^{\frac{1}{rp}q'}dy\right]^{\frac{1}{q'}}\\
 & \leq c_{d}r'\|b\|_{BMO}\left[\frac{1}{|Q|}\int_{Q}\left(\frac{1}{|Q|}\int_{Q}\|W^{\frac{1}{p}}(x)W^{-\frac{1}{p}}(y)\|^{rp}dx\right)^{\frac{1}{rp}q'}dy\right]^{\frac{1}{q'}}
\end{split}
\]
Arguing as above,
\[
\|\||b(x)-\langle b\rangle_{Q}|\|W^{\frac{1}{p}}(x)W^{-\frac{1}{p}}(y)\|\|_{A_{x},Q}\|_{B_{y},Q}\leq c_{d,p,q,n}\|b\|_{BMO}[W]_{A_{q,\infty}^{sc}}\left[W\right]_{A_{q}}^{\frac{1}{p}}.
\]
For the other term we note that choosing $C(t)=t^{rp}$ with $r=1+\frac{1}{2^{d+11}[W]_{A_{q,\infty}^{sc}}}$
and $D(t)=t^{q'}$ we have that $\|M_{\overline{C}}\|_{L^{p'}}\leq c_{d}(r')^{\frac{1}{p'}}\leq c_{n,d}{ [W]_{A_{q,\infty}^{sc}}  ^{\frac{1}{p'}}}$
and $\|M_{\overline{D}}\|_{L^{p}}\leq c_{d,p,q}$. We observe that
\[
\begin{split} & \||b(y)-\langle b\rangle_{Q}|\|\|W^{\frac{1}{p}}(x)W^{-\frac{1}{p}}(y)\|\|_{C_{x},Q}\|_{D_{y},Q}\\
 & =\left[\frac{1}{|Q|}\int_{Q}|b(y)-\langle b\rangle_{Q}|^{q'}\left(\frac{1}{|Q|}\int_{Q}\|W^{\frac{1}{p}}(x)W^{-\frac{1}{p}}(y)\|^{rp}dx\right)^{\frac{1}{rp}q'}dy\right]^{\frac{1}{q'}}.
\end{split}
\]
Taking into account \eqref{eq:RHolderAq}, and arguing as above
\[
\begin{split} & \left[\frac{1}{|Q|}\int_{Q}|b(y)-\langle b\rangle_{Q}|^{q'}\left(\frac{1}{|Q|}\int_{Q}\|W^{\frac{1}{p}}(x)W^{-\frac{1}{p}}(y)\|^{rp}dx\right)^{\frac{1}{rp}q'}dy\right]^{\frac{1}{q'}}\\
 & \leq c_{n}\left[\frac{1}{|Q|}\int_{Q}|b(y)-\langle
 b\rangle_{Q}|^{q'}\left(\frac{1}{|Q|}\int_{Q}\|W^{\frac{1}{q}}(x)W^{-\frac{1}{q}}(y)\|^{q}dx\right)^{\frac{q'}{p}}dy\right]^{\frac{1}{q'}}\\
 & \leq c_{n}\left[\frac{1}{|Q|}\int_{Q}|b(y)-\langle
 b\rangle_{Q}|^{q'{ \left(\frac{p}{q}\right)'}}dy\right]^{\frac{1}{q'{ \left(\frac{p}{q}\right)'}}}\left[\frac{1}{|Q|}\int_{Q}\left(\frac{1}{|Q|}\int_{Q}\|W^{\frac{1}{q}}(x)W^{-\frac{1}{q}}(y)\|^{q}dx\right)^{\frac{q'}{q}}dy\right]^{\frac{q}{q'p}}\\
 & \leq c_{n,d,p,q}\|b\|_{BMO}\left[\frac{1}{|Q|}\int_{Q}\left(\frac{1}{|Q|}\int_{Q}\|W^{\frac{1}{q}}(x)W^{-\frac{1}{q}}(y)\|^{q}dx\right)^{\frac{q'}{q}}dy\right]^{\frac{q}{q'p}}\\
 & \leq c_{n,d,p,q}\|b\|_{BMO}\left[W\right]_{A_{q}}^{\frac{1}{p}}.
\end{split}
\]
Gathering all the choices and estimates above, a direct application
of Lemma \ref{Lem:BumpsComm} yields the desired estimate.

\subsection{Proof of the estimates for Maximal Rough Singular Integrals}

Arguing as in \cite{DiPHL}, we have that
\[
\left\Vert T_{\Omega,W,p}^{*}\vec{f}\right\Vert _{{  L^{p}(\mathbb{R}^{d})}}\lesssim\left\Vert M_{W,p}\vec{f}\right\Vert _{{  L^{p}(\mathbb{R}^{d})}}+\sup_{\||g|\|_{L^{p'}(\mathbb{R}^{d})}=1}\inf_{\varepsilon>0}\sup_{\mathcal{S}}\frac{1}{\varepsilon}\sum_{Q\in\mathcal{S}}|Q|\left\langle \left\langle W^{-\frac{1}{p}}\vec{f}\right\rangle \right\rangle _{1+\varepsilon,Q}\left\langle \left\langle W^{\frac{1}{p}}\vec{g}\right\rangle \right\rangle _{1+\varepsilon,Q}
\]
where we interpret the product in second term as the right endpoint
of the Minkowski product
\[
AB=\left\{ \left(a,b\right):a\in A,b\in B\right\}
\]
which, in the case of ${  A,B\subset\mathbb{R}^{d}}$ being convex symmetric
sets is a closed symmetric interval. The estimate for the first term
is \eqref{eq:A1MW} in the case $q=1$ and \eqref{eq:AqMW} in the
case $q>1$, so we are left with settling the estimate for the second
term. We proceed as follows.

First we notice that if $a\in\left\langle \left\langle W^{-\frac{1}{p}}\vec{f}\right\rangle \right\rangle _{1+\varepsilon,Q}$
and $b\in\left\langle \left\langle W^{\frac{1}{p}}\vec{g}\right\rangle \right\rangle _{1+\varepsilon,Q}$
then
\[
|(a,b)|\leq\frac{1}{|Q|}\int_{Q}\frac{1}{|Q|}\int_{Q}|(W^{{ -\frac{1}{p}}}(x)\vec{f}(x)\varphi_{a,Q}(x),W^{\frac{1}{p}}(y)\vec{g}(y)\psi_{b,Q}(y))|dxdy
\]
where $\varphi_{a,Q},\psi_{b,Q}\in L^{(1+\varepsilon)'}(Q)$. Since
$W^{\frac{1}{p}}$ is positive definite and symmetric a.e. we have
that
\[
\begin{split} & \frac{1}{|Q|}\int_{Q}\frac{1}{|Q|}\int_{Q}\left|\left(W^{-\frac{1}{p}}(x)\vec{f}(x)\varphi_{a,Q}(x),W^{\frac{1}{p}}(y)\vec{g}(y)\psi_{b,Q}(y)\right)\right|dxdy\\
 & =\frac{1}{|Q|}\int_{Q}\frac{1}{|Q|}\int_{Q}\left|\left(W^{\frac{1}{p}}(y)W^{-\frac{1}{p}}(x)\vec{f}(x)\varphi_{a,Q}(x),\vec{g}(y)\psi_{b,Q}(y)\right)\right|dxdy\\
 & \leq\frac{1}{|Q|}\int_{Q}\frac{1}{|Q|}\int_{Q}|W^{\frac{1}{p}}(y)W^{-\frac{1}{p}}(x)\vec{f}(x)||\varphi_{a,Q}(x)||\vec{g}(y)||\psi_{b,Q}(y)|dxdy\\
 & \leq\frac{1}{|Q|}\int_{Q}|\vec{g}(y)||\psi_{b,Q}(y)|\left(\frac{1}{|Q|}\int_{Q}|W^{\frac{1}{p}}(y)W^{-\frac{1}{p}}(x)\vec{f}(x)|^{1+\varepsilon}dx\right)^{\frac{1}{1+\varepsilon}}\left(\frac{1}{|Q|}\int_{Q}|\varphi_{a,Q}(x)|^{\left(1+\varepsilon\right)'}dx\right)^{\frac{1}{(1+\varepsilon)'}}dy\\
 & \leq\frac{1}{|Q|}\int_{Q}|\vec{g}(y)|\left(\frac{1}{|Q|}\int_{Q}|W^{\frac{1}{p}}(y)W^{-\frac{1}{p}}(x)\vec{f}(x)|^{1+\varepsilon}dx\right)^{\frac{1}{1+\varepsilon}}|\psi_{b,Q}(y)|dy\\
 & \leq\left(\frac{1}{|Q|}\int_{Q}|\vec{g}(y)|^{1+\varepsilon}\left(\frac{1}{|Q|}\int_{Q}|W^{\frac{1}{p}}(y)W^{-\frac{1}{p}}(x)\vec{f}(x)|^{1+\varepsilon}dx\right)dy\right)^{\frac{1}{1+\varepsilon}}\left(\frac{1}{|Q|}\int_{Q}|\psi_{b,Q}(y)|^{\left(1+\varepsilon\right)'}dy\right)^{\frac{1}{(1+\varepsilon)'}}\\
 & \leq\left(\frac{1}{|Q|}\int_{Q}\frac{1}{|Q|}\int_{Q}|W^{\frac{1}{p}}(y)W^{-\frac{1}{p}}(x)\vec{f}(x)|^{1+\varepsilon}|\vec{g}(y)|^{1+\varepsilon}dxdy\right)^{\frac{1}{1+\varepsilon}}
\end{split}
\]
Then
\[
\begin{split} & \frac{1}{\varepsilon}\sum_{Q\in\mathcal{S}}|Q|\left\langle \left\langle W^{-\frac{1}{p}}\vec{f}\right\rangle \right\rangle _{1+\varepsilon,Q}\left\langle \left\langle W^{\frac{1}{p}}\vec{g}\right\rangle \right\rangle _{1+\varepsilon,Q}\\
 & \lesssim\frac{1}{\varepsilon}\sum_{Q\in\mathcal{S}}|Q|\left(\frac{1}{|Q|}\int_{Q}\frac{1}{|Q|}\int_{Q}|W^{\frac{1}{p}}(y)W^{-\frac{1}{p}}(x)\vec{f}(x)|^{1+\varepsilon}|\vec{g}(y)|^{1+\varepsilon}dxdy\right)^{\frac{1}{1+\varepsilon}}.
\end{split}
\]
We are going to obtain a suitable control for this term providing
an argument analogous to the one we gave for Calderón-Zygmund operators.
Let $\tau=8\cdot2^{d+11}$,
\[
\varepsilon=\frac{1}{p\frac{p}{p-q}\tau[W]_{A_{q,\infty}^{sc}}}\qquad s=1+\frac{1}{p'\frac{p}{p-q}(\tau-2)[W]_{A_{q,\infty}^{sc}}}
\]
Then we have that
\[
\begin{split} & \left(\frac{1}{|Q|}\int_{Q}\frac{1}{|Q|}\int_{Q}|W^{\frac{1}{p}}(y)W^{-\frac{1}{p}}(x)\vec{f}(x)|^{1+\varepsilon}|\vec{g}(y)|^{1+\varepsilon}dxdy\right)^{\frac{1}{1+\varepsilon}}\\
 & \leq\left(\frac{1}{|Q|}\int_{Q}|\vec{f}(x)|^{1+\varepsilon}\frac{1}{|Q|}\int_{Q}\|W^{\frac{1}{p}}(y)W^{-\frac{1}{p}}(x)\|^{1+\varepsilon}|\vec{g}(y)|^{1+\varepsilon}dydx\right)^{\frac{1}{1+\varepsilon}}\\
 & \leq\left(\frac{1}{|Q|}\int_{Q}|\vec{f}(x)|^{1+\varepsilon}\left(\frac{1}{|Q|}\int_{Q}\|W^{\frac{1}{p}}(y)W^{-\frac{1}{p}}(x)\|^{ps(1+\varepsilon)}dy\right)^{\frac{1}{ps(1+\varepsilon)}(1+\varepsilon)}dx\right)^{\frac{1}{1+\varepsilon}}M_{(ps)'(1+\varepsilon)}(|\vec{g}|)(x).
\end{split}
\]
{ It is not hard to check that
\begin{equation}
\|M_{(ps)'(1+\varepsilon)}\|_{L^{p'}(\mathbb{R}^{d})}\lesssim\left([W]_{A_{q,\infty}^{sc}}\right)^{\frac{1}{p'}}\qquad1\leq q<p.\label{eq:MaxLpPrime}
\end{equation}
Indeed, taking into account \eqref{eq:ConstMA} it suffices to prove that
\[\frac{1}{p'-(ps)'(1+\varepsilon)}\lesssim [W]_{A_{q,\infty}^{sc}}.\]
First we note that
\[\frac{1}{p'-(ps)'(1+\varepsilon)}=\frac{1}{p}\frac{(ps-1)(p-1)}{(ps-1)-s(p-1)(1+\varepsilon)}.\]
Working on the denominator we have that
\[
\begin{split} & (ps-1)-s(p-1)(1+\varepsilon)\\
 & =(ps-1)+(-sp+s)(1+\varepsilon)=ps-1-ps-ps\varepsilon+s+s\varepsilon\\
 & =-1-ps\varepsilon+s+s\varepsilon=-1+((1-p)\varepsilon+1)\left(1+(p-1)\frac{\tau}{\tau-2}\varepsilon\right)\\
 & =-1-(p-1)^{2}\frac{\tau}{\tau-2}\varepsilon^{2}+(p-1)\frac{\tau}{\tau-2}\varepsilon-(p-1)\varepsilon+1\\
 & =-(p-1)^{2}\frac{\tau}{\tau-2}\varepsilon^{2}+(p-1)\frac{\tau}{\tau-2}\varepsilon-(p-1)\varepsilon\\
 & =(p-1)\varepsilon\left[-(p-1)\frac{\tau}{\tau-2}\varepsilon+\frac{\tau}{\tau-2}-\frac{\tau-2}{\tau-2}\right]\\
 & =\frac{(p-1)\varepsilon}{\tau-2}\left[2-(p-1)\tau\varepsilon\right]
\end{split}
\]
It is clear that $(p-1)\tau\varepsilon\leq1$. Combining this estimate with the identities above,
\[
\begin{split}\frac{1}{p'-(ps)'(1+\varepsilon)}= & \frac{1}{p}\frac{(ps-1)(p-1)}{(ps-1)-s(p-1)(1+\varepsilon)}=\frac{\tau-2}{p}\frac{(ps-1)(p-1)}{\varepsilon(p-1)\left[2-(p-1)\tau\varepsilon\right]}\\
 & =\frac{\tau-2}{\varepsilon p}\frac{(ps-1)}{\left[2-(p-1)\tau\varepsilon\right]}\leq(ps-1)\frac{\tau-2}{\varepsilon p}\lesssim[W]_{A_{q,\infty}^{sc}}
\end{split}
\]}
Now we focus on the proof of \eqref{eq:A1Rough}. In that case, by
Lemma \ref{Lem:Key}, since $(1+\varepsilon)s\leq1+\frac{1}{2^{d+11}[W]_{A_{q,\infty}^{sc}}}$,
\[
\begin{split} & \left(\frac{1}{|Q|}\int_{Q}|\vec{f}(x)|^{1+\varepsilon}\left(\frac{1}{|Q|}\int_{Q}\|W^{\frac{1}{p}}(y)W^{-\frac{1}{p}}(x)\|^{ps(1+\varepsilon)}dy\right)^{\frac{1}{ps(1+\varepsilon)}(1+\varepsilon)}dx\right)^{\frac{1}{1+\varepsilon}}M_{(ps)'(1+\varepsilon)}(|\vec{g}|)(x)\\
 & \leq\left(\frac{1}{|Q|}\int_{Q}|\vec{f}(x)|^{1+\varepsilon}\left(\frac{1}{|Q|}\int_{Q}\|W(y)W^{-1}(x)\|dy\right)^{\frac{1}{p}(1+\varepsilon)}dx\right)^{\frac{1}{1+\varepsilon}}M_{(ps)'(1+\varepsilon)}(|\vec{g}|)(x)\\
 & \leq[W]_{A_{1}}^{\frac{1}{p}}M_{1+\varepsilon}(|\vec{f}|)(x)M_{(ps)'(1+\varepsilon)}(|\vec{g}|)(x)
\end{split}
\]
This yields that
\[
\begin{split} & \frac{1}{\varepsilon}\sum_{Q\in\mathcal{S}}{  |Q|} \left(\frac{1}{|Q|}\int_{Q}\frac{1}{|Q|}\int_{Q}|W^{\frac{1}{p}}(y)W^{-\frac{1}{p}}(x)\vec{f}(x)|^{1+\varepsilon}|\vec{g}(y)|^{1+\varepsilon}dxdy\right)^{\frac{1}{1+\varepsilon}}.\\
 & \lesssim[W]_{A_{1,\infty}^{sc}} { [W]_{A_{1}}^{\frac{1}{p}}} \sum_{Q\in\mathcal{S}}|Q|M_{1+\varepsilon}(|\vec{f}|)(x)M_{(ps)'(1+\varepsilon)}(|\vec{g}|)(x)\\
 & \lesssim[W]_{A_{1,\infty}^{sc}}[W]_{A_{1}}^{\frac{1}{p}}\int_{\mathbb{R}^{d}}M_{1+\varepsilon}(|\vec{f}|)M_{(ps)'(1+\varepsilon)}(|\vec{g}|)(x)\\
 & \lesssim[W]_{A_{1,\infty}^{sc}}[W]_{A_{1}}^{\frac{1}{p}}\|M_{1+\varepsilon}\|_{L^{p}(\mathbb{R}^{d})}\|M_{(ps)'(1+\varepsilon)}\|_{L^{p'}(\mathbb{R}^{d})}{ \|\vec{f}\|_{ L^{p}(\mathbb{R}^{d};\mathbb{C}^{n})}\|\vec{g}\|_{L^{p'}(\mathbb{R}^{d};\mathbb{C}^{n})}}\\
 & \lesssim[W]_{A_{1,\infty}^{sc}}^{1+\frac{1}{p'}}[W]_{A_{1}}^{\frac{1}{p}}{ \|\vec{f}\|_{L^{p}(\mathbb{R}^{d};\mathbb{C}^{n})}\|\vec{g}\|_{L^{p'}(\mathbb{R}^{d};\mathbb{C}^{n})}}
\end{split}
\]
by \eqref{eq:MaxLpPrime} and taking into account that from \eqref{eq:ConstMA}
it follows that
\[
\|M_{1+\varepsilon}\|_{L^{p}(\mathbb{R}^{d})}\lesssim\left(p'\right)^{\frac{1}{p}}.
\]
In the case $q>1$, namely, to settle \eqref{eq:AqRough}, we argue
as follows. By Lemma \ref{Lem:Key}, since $(1+\varepsilon)s\leq1+\frac{1}{2^{d+11}[W]_{A_{q,\infty}^{sc}}}$
\[
\begin{split} & \left(\frac{1}{|Q|}\int_{Q}|\vec{f}(x)|^{1+\varepsilon}\left(\frac{1}{|Q|}\int_{Q}\|W^{\frac{1}{p}}(y)W^{-\frac{1}{p}}(x)\|^{ps(1+\varepsilon)}dy\right)^{\frac{1}{ps(1+\varepsilon)}(1+\varepsilon)}dx\right)^{\frac{1}{1+\varepsilon}}M_{(ps)'(1+\varepsilon)}(|{ \V{g}}|)(x)\\
 & \leq\left(\frac{1}{|Q|}\int_{Q}|\V{f}(x)|^{1+\varepsilon}\left(\frac{1}{|Q|}\int_{Q}\|W^{\frac{1}{q}}(y)W^{-\frac{1}{q}}(x)\|^{q}dy\right)^{\frac{1}{p}(1+\varepsilon)}dx\right)^{\frac{1}{1+\varepsilon}}M_{(ps)'(1+\varepsilon)}(|{ \V{g}}|)(x)
\end{split}
\]
{  If we call $W=V^{1-q}$, then}
\[
\begin{split} & \left(\frac{1}{|Q|}\int_{Q}|{ \V{f}}(x)|^{1+\varepsilon}\left(\frac{1}{|Q|}
\int_{Q}\|V^{-\frac{1}{q'}}(y)V^{\frac{1}{q'}}(x)\|^{q}dy\right)^{\frac{1}{p}(1+\varepsilon)}dx\right)^{\frac{1}{1+\varepsilon}}
M_{(ps)'(1+\varepsilon)}(|{ \V{g}}|)(x)\\
 & \leq\left(\left(\frac{1}{|Q|}\int_{Q}\|V^{-\frac{1}{q'}}(y)V^{\frac{1}{q'}}(x)\|^{q}dy\right)^{\frac{q'}{q}}dx\right)^{\frac{q}{q'p}}
 M_{\left(\frac{pq'}{q(1+\varepsilon)}\right)^{'}(1+\varepsilon)}(|{ \V{f}}|)(x)M_{(ps)'(1+\varepsilon)}(|{ \V{g}}|)(x)\\
  & \leq[V]_{A_{q'}}^{\frac{q-1}{p}}M_{\left(\frac{pq'}{q(1+\varepsilon)}\right)^{'}(1+\varepsilon)}(|{ \V{f}}|)(x)
  M_{(ps)'(1+\varepsilon)}(|{ \V{g}}|)(x)\\
 & \lesssim[W]_{A_{q}}^{\frac{1}{p}}M_{\left(\frac{pq'}{q(1+\varepsilon)}\right)^{'}(1+\varepsilon)}(|{ \V{f}}|)(x)
 M_{(ps)'(1+\varepsilon)}(|{ \V{g}}|)(x)
\end{split}
\]
where the last estimate follows from Proposition \ref{Prop:AqAqprime}.
Taking that estimate into account,
\[
\begin{split} & \frac{1}{\varepsilon}\sum_{Q\in\mathcal{S}}{ {|Q| }} \left(\frac{1}{|Q|}\int_{Q}\frac{1}{|Q|}\int_{Q}|W^{\frac{1}{p}}(y)W^{-\frac{1}{p}}(x){ \V{f}}(x)|^{1+\varepsilon}|{ \V{g}}(y)|^{1+\varepsilon}dxdy\right)^{\frac{1}{1+\varepsilon}}.\\
 & \lesssim[W]_{A_{q,\infty}^{sc}}[W]_{A_{q}}^{\frac{1}{p}}\sum_{Q\in\mathcal{S}}|Q|M_{\left(\frac{pq'}{q(1+\varepsilon)}\right)^{'}(1+\varepsilon)}{ (|\V{f}|)}(x)M_{(ps)'(1+\varepsilon)}{ (|\V{g}|)}(x)\\
 & \lesssim[W]_{A_{q,\infty}^{sc}}[W]_{A_{q}}^{\frac{1}{p}}\int_{\mathbb{R}^{d}}M_{\left(\frac{pq'}{q(1+\varepsilon)}\right)^{'}(1+\varepsilon)}(|{ \V{f}}|)(x)M_{(ps)'(1+\varepsilon)}(|{ \V{g}}|)dx\\
 & \lesssim[W]_{A_{q,\infty}^{sc}}[W]_{A_{q}}^{\frac{1}{p}}\left\Vert M_{\left(\frac{pq'}{q(1+\varepsilon)}\right)^{'}(1+\varepsilon)}\right\Vert _{L^{p}(\mathbb{R}^{d})}\|M_{(ps)'(1+\varepsilon)}\|_{L^{p'}(\mathbb{R}^{d})}{  \|\vec{f}\|_{L^{p}(\mathbb{R}^{d};\mathbb{C}^{n})}\|\vec{g}\|_{L^{p'}(\mathbb{R}^{d};\mathbb{C}^{n})}}\\
 & \lesssim[W]_{A_{q,\infty}^{sc}}^{1+\frac{1}{p'}}[W]_{A_{q}}^{\frac{1}{p}}{  \|\vec{f}\|_{L^{p}(\mathbb{R}^{d};\mathbb{C}^{n})}\|\vec{g}\|_{L^{p'}(\mathbb{R}^{d};\mathbb{C}^{n})}}
\end{split}
\]
by \eqref{eq:MaxLpPrime} and taking into account that from \eqref{eq:ConstMA}
it is not hard to derive that
\[
\left\Vert M_{\left(\frac{pq'}{q(1+\varepsilon)}\right)^{'}(1+\varepsilon)}\right\Vert _{L^{p}(\mathbb{R}^{d})}\lesssim\left(\frac{p}{p-q}\right)^{\frac{1}{p}}.
\]
This ends the proof of \eqref{eq:AqRough}.

\section*{Acknowledgements}

We would like to thank Francesco Di Plinio for suggesting us to address
the problem of the maximal rough singular integral.

The second author would like to express his gratitude to the Department
of Mathematics of Lund University for the hospitality shown during
his visit between September and November 2017.

\end{document}